\documentclass[11pt]{amsart}
\usepackage{amsmath,amssymb,mathrsfs,pstricks}
\usepackage{tikz-cd}
\usepackage[colorlinks=true,linkcolor=black,citecolor=black,urlcolor=black]{hyperref}

\psset{unit=1pt}
\psset{arrowsize=4pt 1}
\psset{linewidth=.5pt}

\newgray{lightgray}{.80}

\countdef\x=23
\countdef\y=24
\countdef\z=25
\countdef\t=26

\def\graybox(#1,#2){
\x=#1 \y=#2
\z=\x \t=\y
\advance\z by 10
\advance\t by 10
\psframe[fillstyle=solid,fillcolor=lightgray,linewidth=0pt](\x,\y)(\z,\t)
\psline[linewidth=.5pt](\x,\y)(\x,\t)(\z,\t)(\z,\y)(\x,\y)}

\def\whitebox(#1,#2){
\x=#1 \y=#2
\z=\x \t=\y
\advance\z by 10
\advance\t by 10
\psframe[fillstyle=solid,fillcolor=white,linewidth=0pt](\x,\y)(\z,\t)
\psline[linewidth=.5pt](\x,\y)(\x,\t)(\z,\t)(\z,\y)(\x,\y)}

\newcommand{\define}{\textbf}

\newcommand{\excise}[1]{}

\newcommand{\isom}{\cong}
\renewcommand{\setminus}{\smallsetminus}
\renewcommand{\phi}{\varphi}
\newcommand{\exterior}{\textstyle\bigwedge}
\renewcommand{\tilde}{\widetilde}
\renewcommand{\hat}{\widehat}
\renewcommand{\bar}{\overline}

\newcommand{\C}{\mathbb{C}}
\renewcommand{\P}{\mathbb{P}}

\newcommand{\Z}{\mathbb{Z}}

\newcommand{\KK}{\mathcal{K}}
\newcommand{\OO}{\mathcal{O}}

\newcommand{\liea}{\mathfrak{a}}
\newcommand{\lieb}{\mathfrak{b}}
\newcommand{\lieg}{\mathfrak{g}}
\newcommand{\lieh}{\mathfrak{h}}
\newcommand{\lien}{\mathfrak{n}}

\newcommand{\liet}{\mathfrak{t}}
\newcommand{\univ}{\mathfrak{U}}

\newcommand{\liegl}{\mathfrak{gl}}
\newcommand{\liesl}{\mathfrak{sl}}
\newcommand{\liesp}{\mathfrak{sp}}
\newcommand{\lieso}{\mathfrak{so}}

\newcommand{\Gr}{{Gr}}
\newcommand{\QQ}{\mathcal{Q}}

\newcommand{\GGr}{\mathcal{G}r}

\newcommand{\pt}{\mathrm{pt}}
\newcommand{\htpt}{S}

\newcommand{\e}{\varepsilon}

\newcommand{\barf}{\bar{f}}

\newcommand{\ee}{\mathrm{e}}

\newcommand{\bx}{\footnotesize{\fbox{}}}
\newcommand{\spin}{\mathbb{S}}

\DeclareMathOperator{\codim}{codim}

\DeclareMathOperator{\Span}{span}

\DeclareMathOperator{\Sym}{Sym}

\DeclareMathOperator{\Pf}{Pf}
\renewcommand{\det}{\mathrm{det}}

\newtheorem{theorem}{Theorem}
\newtheorem{lemma}[theorem]{Lemma}

\newtheorem*{thm*}{Theorem}

\theoremstyle{definition}

\newtheorem{remark}[theorem]{Remark}
\newtheorem{example}[theorem]{Example}

\begin{document}

\title[Schubert calculus and the Satake correspondence]{Minuscule Schubert calculus and the geometric Satake correspondence}
\author{Dave Anderson}
\address{Department of Mathematics, The Ohio State University, Columbus, OH 43210}
\email{anderson.2804@math.osu.edu}
\author{Antonio Nigro}
\address{Universidade Federal Fluminense, Niter\'{o}i, RJ}
\email{antonio.nigro@gmail.com}
\keywords{geometric Satake correspondence, Schubert calculus, affine Grassmannian, Pfaffian, equivariant quantum cohomology}
\date{January 22, 2021}
\thanks{DA was partially supported by a postdoctoral fellowship from the Instituto Nacional de Matem\'atica Pura e Aplicada (IMPA) and by NSF Grant DMS-1502201.}

\begin{abstract}
We describe a relationship between work of Gatto, Laksov, and their collaborators on realizations of (generalized) Schubert calculus of Grassmannians, and the geometric Satake correspondence of Lusztig, Ginzburg, and Mirkovi\'c and Vilonen.  Along the way we obtain new proofs of equivariant Giambelli formulas for the ordinary and orthogonal Grassmannians, as well as a simple derivation of the ``rim-hook'' rule for computing in the equivariant quantum cohomology of the Grassmannian.
\end{abstract}

\maketitle

\section{Introduction}

The goal of this article is to illustrate connections between several circles of ideas in Schubert calculus, representation theory, and symmetric functions.  The driving force behind the connections we explore is the {\em geometric Satake correspondence}---which, in the special cases we examine, matches Schubert classes in the Grassmannian (and related spaces) with weight vectors in exterior products (and related representations).  We will focus especially on an equivariant version of this correspondence, where additional bases appear on both sides.  An analysis of the transition matrices for these bases leads us directly to well-known symmetric functions: the Schur $S$- and $P$-functions, along with their factorial generalizations.

The approach we take is primarily expository, at least in the sense that all the main theorems have appeared previously.  However, in describing a relationship among notions which do not often appear side by side, we obtain some novel consequences: new and simple proofs of the equivariant Giambelli formulas for ordinary and orthogonal Grassmannians, as well as of a rule for computing in the equivariant quantum cohomology of the Grassmannian.  This perspective should find further applications within Schubert calculus.  One theme we wish to emphasize is this: when some aspect of $H_T^*\Gr(k,n)$ appears related to exterior algebra, one can expect generalizations via the Satake correspondence, either to other minuscule spaces, or to other subvarieties of the affine Grassmannian.

Our starting point is a very simple observation, which must be quite old.  The Grassmannian $\Gr(k,n)$ has a decomposition into Schubert cells, indexed by $k$-element subsets of $[n] := \{1,\ldots,n\}$.  As $\C$-vector spaces, therefore, one has
\begin{equation}\label{e.obs1}
  H_*(\Gr(k,n),\C) = H^*(\Gr(k,n),\C) = \exterior^k \C^{n},
\end{equation}
and the cohomology ring $H^*(\Gr(k,n),\C)$ acts as a certain ring of operators on the exterior power.  The rest of this story is an attempt to add as much structure as possible to this identification.

As a first step, for $k=1$, let us identify the basis of linear subspaces of $\P^{n-1}$ with the standard basis of $\C^n$:
\[
  [\P^{n-i}] = \e_i \quad \text{ in }\quad H_{n-i}\P^{n-1} = H^{i-1}\P^{n-1}.
\]
(Since the spaces we consider have no odd-degree singular (co)homology, we will economize notation by always writing $H_i$ and $H^i$ for singular homology and cohomology in degree $2i$.)  This induces a grading on $\C^n$, and makes \eqref{e.obs1} an isomorphism of graded vector spaces.  More generally, writing $\Omega_I$ for the Schubert variety corresponding to $I = \{i_1 < \cdots < i_k\} \subseteq [n]$, we identify
\[
  [\Omega_I] = \e_I := \e_{i_1} \wedge \cdots \wedge \e_{i_k}
\]
in $H_*\Gr(k,n)=H^*\Gr(k,n)$.

To be a little more specific, since it will matter later, these are the {\it opposite} Schubert varieties: $\Omega_I$ is the closure of the cell $\Omega^\circ_I$ whose representing $k\times n$ matrices have pivots in columns $I$, with zeroes to the left of the pivots.  For example, in $\Gr(3,8)$ we have
\[
  \Omega^\circ_{\{2,4,7\}} = \left[\begin{array}{cccccccc}
  0 & 1 & * & 0 & * & * & 0 & * \\
  0 & 0 & 0 & 1 & * & * & 0 & * \\
  0 & 0 & 0 & 0 & 0 & 0 & 1 & *\end{array}\right].
\]
There is a standard bijection between $I \subseteq [n]$ and partitions $\lambda$ fitting inside the $k\times(n-k)$ rectangle; one puts
\[
  \lambda_{k+1-j} = i_j - j.
\]
Then, writing $\Omega_\lambda = \Omega_I$, the grading is realized by $\codim \Omega_\lambda = |\lambda| = \lambda_1+\cdots+\lambda_k$.  We will write $\sigma_\lambda$ for the (co)homology class $[\Omega_\lambda]$.

The essential idea is to exploit the isomorphism \eqref{e.obs1} and use linear algebra to develop the basic ingredients of Schubert calculus---including Giambelli and Pieri formulas.  This can be done using elementary (though still nontrivial) methods.

Again, the first step is to examine the easiest case, where $k=1$.  The action of the divisor class $\sigma_{\bx}$ translates into an operator $\xi$ on $\C^n$, given by
\[
  \xi\cdot \e_i = \begin{cases} \e_{i+1} & \text{ if } i<n; \\ 0 &\text{ if }i=n.\end{cases}
\]
So this has matrix
\begin{equation}\label{e.xi}
\xi = \left[\begin{array}{cccc}
0 &  &  &  \\
1 & \ddots &  &  \\
0 & \ddots & \ddots &  \\
0 & 0 & 1 & 0\end{array}\right]
\end{equation}
and one can view it as an element of the Lie algebra $\liegl_n$ (or $\liesl_n$).  This Lie algebra acts naturally on exterior powers $\exterior^k\C^n$, and the basic observation is that for any $k$, $\sigma_{\bx}$ acts on $H^*\Gr(k,n)$ (via cup product) just as the matrix $\xi$ acts on $\exterior^k\C^n$ (via the Lie algebra action).

For example, consider $\sigma_{1}$ acting on $H^*\Gr(2,n)$.  We have
\begin{align*}
  \xi\cdot \e_1\wedge\e_2 &= \e_2\wedge\e_2 + \e_1\wedge\e_3 \\
   & = \e_1\wedge\e_3,
\end{align*}
corresponding to $\sigma_{\bx} \cdot \sigma_\emptyset = \sigma_{(1)}$.  Similarly,
\begin{align*}
  \xi\cdot \e_1\wedge\e_3 &= \e_2\wedge\e_3 + \e_1\wedge\e_4,
\end{align*}
corresponding to $\sigma_{\bx}\cdot\sigma_{(1)} = \sigma_{(1,1)}+\sigma_{(2)}$.  Modelling the cohomology of $\Gr(2,n)$ using symmetric functions, the divisor class corresponds to the sum of variables $p_1$.  However, that the Lie algebra action of the matrix $\xi^2$ does not correspond to multiplication by $(\sigma_{\bx})^2$.  For instance, the Lie algebra action gives
\begin{align*}
 \xi^2 \cdot \e_2 \wedge \e_3 & = \e_4 \wedge \e_3 + \e_2 \wedge \e_5 \\
                             &= \e_2 \wedge \e_5 - \e_3 \wedge \e_4,
\end{align*}
which corresponds to the computation $p_2\cdot \sigma_{(1,1)} = \sigma_{(3,1)} - \sigma_{(2,2)}$, where $p_2$ is the {\em power sum} symmetric polynomial.  (In general, $\xi^k$ will correspond to multiplying by the sum of $k$th powers $p_k$, not $(p_1)^k$.)

So we see a correspondence between the action of a certain element on representations of a Lie algebra, and the multiplication by a divisor class in the cohomology of Grassmannians.  The question thus arises: are the above calculations merely coincidental?  In the next two sections, we will see that they are not, by viewing them as a shadow of the geometric Satake correspondence, a major construction in modern representation theory.

In \S\ref{s.giambelli}, we show how the equivariant Giambelli formula---which computes a Schubert class in $H_T^*\Gr(k,n)$ as a factorial Schur polynomial $s_\lambda(x|t)$---follows directly from the defining properties of the exterior power, putting computations of Gatto, Laksov, Thorup, and others into the general context of the Satake correspondence.  Having done this, it is natural to proceed to the minuscule spaces of type D: in \S\ref{s.spinors}, we apply similar methods to compute equivariant Schubert classes via the factorial Schur $P$-functions $P_\lambda(x|t)$, using computations on even-dimensional quadrics (carried out in \S\ref{s.quadrics}) in place of projective spaces.  The functions $s_\lambda(x|t)$ are frequently defined as a certain ``Jacobi-Trudi'' determinant, but they may also be written as a ratio of two determinants (as was originally done by Cauchy); similarly, $P_\lambda(x|t)$ may be written either as a Pfaffian or as a ratio of two Pfaffians.  A curious aspect of our arguments is that the ratio description of these functions appears naturally, in contrast to most geometric arguments (dating to Giambelli), where the Jacobi-Trudi formulation is used.

We turn to quantum cohomology in \S\ref{s.rim-hook}, where we give a short proof of the equivariant rim-hook rule for computing in $QH_T^*\Gr(k,n)$.  Here the Satake isomorphism serves only a psychological function, and is not logically necessary: the main point is that the combinatorial operation of removing a rim hook from a partition (and picking up a corresponding sign) is precisely that of reducing the indices of a pure wedge modulo $n$.

Most of these ideas have appeared in the work of other authors, at least in some form; as mentioned above, our primary aim is to indicate connections and extract a few new consequences.  We first learned of the possibility of a ``formal'' Schubert calculus on exterior powers from a series of papers by Gatto and Laksov-Thorup in the 2000s, and this point of view has been developed further by these authors and their collaborators \cite{gatto,gatto-salehyan1,gatto-salehyan,gs,laksov,lt1,lt,lt3}.   More detailed references are given throughout the article, and we  point to further connections in a closing section (\S\ref{s.closing}).

\medskip

{\it Acknowledgements.}  This note is partly based on a talk given by the first author at a conference dedicated to the memory of Dan Laksov (Mittag-Leffler, June 2014).  The ideas grew out of conversations we had with Roi Docampo at IMPA.  We learned about the connection between quantum Schubert calculus and the Satake correspondence from a remarkable preprint of Golyshev and Manivel \cite{gm}, and the debt we owe to their work should be evident.  We also thank Reimundo Heluani and Joel Kamnitzer for helping us understand the geometric Satake correspondence.  Finally, we thank the referees for a very careful reading and many thoughtful comments.

%

\section{The geometric Satake correspondence}

A second simple observation is the following: On one hand, the vector space $\exterior^k\C^n$ is a fundamental (and in fact, minuscule) representation $V_{\varpi_k}$ of $\liesl_n$.  On the other hand, $Gr(k,n) = PGL_n/P_{\varpi_k}$, where $P_{\varpi_k}$ is the parabolic corresponding\footnote{In general, a cocharacter $\varpi\colon \C^* \to T \subseteq G$ determines a reductive subgroup $G_\varpi\subseteq G$, the centralizer of its image; the corresponding parabolic $P_\varpi$ is generated by $G_\varpi$ together with the Borel.} to the cocharacter $\varpi_k$ of $PGL_n$.  In fact, if $V=\C^n$ is the standard representation of $\liesl_n$, it is most natural to regard $Gr(k,n) = Gr(k,V^*)$ as parametrizing $k$-planes in the dual vector space.

Work from the 1990's by Ginzburg \cite{ginzburg} and Mirkovi\'c and Vilonen \cite{mv}---which in turn builds on work of Lusztig \cite{lusztig} from the early 1980's---puts this into a more general context.  To describe it we need some terminology and basic facts.

Any reductive group $G$ with maximal torus $T$ comes with a root datum $R$.  Root data have a built-in duality, and exchanging $R$ with $R^\vee$ yields a {\em Langlands dual group} $G^\vee$.  The details will not be too important for now, beyond this: the roots of $G^\vee$ are the coroots of $G$, and the characters of the maximal torus $T^\vee \subseteq G^\vee$ are the co-characters of $T\subseteq G$.  For example, $(PGL_n)^\vee \isom SL_n$, $(Sp_{2n})^\vee \isom SO_{2n+1}$, $(PSO_{2n})^\vee = Spin_{2n}$, and $(GL_n)^\vee \isom GL_n$.  (A significant part of Ginzburg and Mirkovi\'c-Vilonen's program was to give a more intrinsic construction of $G^\vee$.)

Let $\KK = \C((z))$ and $\OO=\C[\![z]\!]$.  The {\em affine Grassmannian} of a complex reductive group $G$ is the infinite-dimensional orbit space
\[
  \GGr_G = G( \KK ) / G( \OO ),
\]
topologized as an ind-variety.  The essence of geometric Satake is to relate the geometry of $\GGr_{G}$ with the representation theory of $G^\vee$.  We will describe a very small part of this correspondence which suffices for our purposes.

The group $G(\OO)$ acts on $\GGr_{G}$ via left multiplication, and its orbits are naturally parametrized by {\it dominant co-characters} $\varpi\colon \C^*\to T$.  These, by duality, are the same as dominant characters of $T^\vee$.  Since there is a well-known indexing of irreducible representations of $G^\vee$ by dominant characters, we have a bijection of sets
\[
\begin{array}{ccc}
\left\{ G(\OO)\text{-orbit closures in } \GGr_{G} \right\} & \leftrightarrow & \left\{ \text{irreducible representations of } G^\vee \right\} \\
\overline{\GGr^\varpi}  & \leftrightarrow & V_\varpi
\end{array}
\]
and as before the goal is to endow this with more structure.

Let $\lieg^\vee=\mathrm{Lie}(G^\vee)$, and take a regular nilpotent element $\xi \in \lieg^\vee$.  (Up to conjugation, in $\liesl_n$ such an element is the matrix from \eqref{e.xi}.  More generally, one can write $\xi = \sum a_i E_{\alpha_i}$ as a sum of simple root vectors.)  Let $\liea \subseteq \lieg^\vee$ be the centralizer of $\xi$, an abelian Lie subalgebra of dimension equal to the rank of $\lieg$.  (In the case of $\xi\in\liesl_n$, this subalgebra is spanned by the matrix powers $\xi,\xi^2,\ldots,\xi^{n-1}$.)  Its universal enveloping algebra, denoted $\univ(\liea)$, acts naturally on any representation of $\liea$.  Since $\liea$ is abelian, $\univ(\liea)$ is just the polynomial algebra $\Sym_\C^*\liea$.

Finally, $IH_*X$ denotes the (middle-perversity) intersection homology of a space $X$, with coefficients in $\C$.  This is a graded vector space which exhibits Poincar\'e duality, and which comes with an action of $H^*Y$, for any $X \to Y$, via cap product.

\begin{theorem}[Geometric Satake \cite{ginzburg,mv}] \label{t.satake}
There are graded isomorphisms of algebras
\[
  H^*(\GGr_{G}^\circ) \isom \univ(\liea) = \Sym_\C^*\liea
\]
and vector spaces
\[
  IH_*(\overline{\GGr^\varpi}) \isom V_\varpi,
\]
for all dominant $\varpi$, and these isomorphisms are compatible with the natural actions of $H^*(\GGr_{G})$ on $IH_*(\overline{\GGr^\varpi})$ (via cap product) and of $\univ(\liea)$ on $V_\varpi$ (via the representation of $\lieg^\vee$).  Furthermore, there is a natural basis of {\it MV-cycles} in $IH_*(\overline{\GGr^\varpi})$ which corresponds to a weight basis of $V_\varpi$.
\end{theorem}

\noindent
The statements proved by Ginzburg and Mirkovi\'c-Vilonen are vastly stronger: they establish an equivalence of tensor categories between the category of $G(\OO)$-equivariant perverse sheaves and the representation category of $G^\vee$.  We will not need this level of generality, however.

The connected components of $\GGr_{G}$ are indexed by elements of $\pi_1(G)$, and in fact there is a natural group structure on the set of components.  In the statement of the theorem, $\GGr_{G}^\circ$ means the identity component, although in fact all components are isomorphic---as spaces, but not compatibly with the left $G(\OO)$-action.

In general, the orbit closure $\overline{\GGr^\varpi}$ is singular, hence the appearance of intersection homology.  However, the minimal orbit in each connected component of $\GGr_{G}$ is closed, so such $\overline{\GGr^\varpi}=\GGr^\varpi$ are smooth, and one has $IH_*=H_*=H^*$.  When $G$ is adjoint (so $G^\vee$ is simply connected), these minimal orbits correspond to the {\it minuscule weights} of $G^\vee$.  For a minuscule weight $\varpi$, one has $\GGr^\varpi = G/P_\varpi$.  Furthermore, in this case the MV-cycles are precisely the Schubert varieties in $G/P_\varpi$ (as noted in \cite[\S1]{kamnitzer}).

\begin{example}
The minuscule weights of $\liesl_n$ are $0,\varpi_1,\ldots,\varpi_{n-1}$, corresponding to the $n$ elements of $\pi_1(PGL_n) = \Z/n\Z$.  The representations are the exterior powers $\exterior^k\C^n$, and the orbits are the Grassmannians $PGL_n/P_{\varpi_k} = \Gr(k,n)$, for $0\leq k\leq n-1$.

The minuscule weights of $\lieso_{2n}$ are $0,\varpi_1,\varpi_{n-1},\varpi_n$, where the nonzero ones are the fundamental weights corresponding to the three end-nodes of the $D_n$ Dynkin diagram.  The representations are the standard one, $V_{\varpi_1} = \C^{2n}$, and the half-spin representations, $V_{\varpi_{n-1}}=\spin_n^+$ and $V_{\varpi_n}=\spin_n^-$.  The orbits are the quadric $\QQ^{2n-2} = PSO_{2n}/P_{\varpi_1}$, and the two maximal orthogonal Grassmannians $OG^+(n,2n) = PSO_{2n}/P_{\varpi_{n-1}}$ and $OG^-(n,2n) = PSO_{2n}/P_{\varpi_n}$.

In type $C_n$, the weight $\varpi_1$ corresponding to the standard representation of $\liesp_{2n}$ is minuscule, and in type $B_n$, the weight $\varpi_n$ corresponding to the spin representation of $\lieso_{2n+1}$ is minuscule.  The minimal orbits are isomorphic to $\P^{2n-1}$ and $OG^+(n+1,2n+2)$, respectively, so they already occur in types $A$ and $D$.

Among simple groups $G$, there are only a few other instances of nonzero minuscule weights.  In type $E_6$, the weights $\varpi_1$ and $\varpi_6$ are minuscule, corresponding to the $27$-dimensional  Jordan algebra representation and its dual; the corresponding varieties are the octonionic projective plane and its dual.  In type $E_7$, there is one nonzero minuscule weight, whose corresponding representation is the $56$-dimensional Brown algebra, and whose corresponding homogeneous space is known as the Freudenthal variety.
\end{example}

As a special case of Theorem~\ref{t.satake}, we have the isomorphism
\[
  H_*\Gr(k,n) = \exterior^k\C^n,
\]
together with the compatible actions by divisor class and regular nilpotent, described in the introduction---in particular, it is no coincidence that one can do this.  We will push this further to obtain a new perspective on Laksov's computation of the equivariant cohomology of $\Gr(k,n)$ in \S\ref{s.giambelli}.

\begin{remark}
Let $\Lambda_\C$ be the ring of symmetric functions with coefficients in $\C$.  It can be identified with the infinite polynomial ring $\C[p_1,p_2,\ldots]$, where $p_r = x_1^r + x_2^r + \cdots$ is the power sum symmetric function.  For $G=PGL_n$, Bott \cite{bott} showed that there is a natural map
\[
  \Lambda_\C \to H^*\GGr_{G}^\circ,
\]
identifying the RHS as $\Lambda_\C/(p_n,p_{n+1},\ldots) \isom \C[p_1,\ldots,p_{n-1}]$.  Furthermore, this identifies the subspace $\mathfrak{P} \subseteq H^*\GGr_{G}^\circ$ of {\em primitive classes} with the space spanned by $\{p_1,\ldots,p_{n-1}\}$.  Ginzburg's proof of the first part of Theorem~\ref{t.satake} establishes an isomorphism $\liea \isom \mathfrak{P}$.  So the power sum symmetric functions play a central role in this story; we will see an echo of this in the rim-hook rule for quantum cohomology (\S\ref{s.rim-hook}).  (A word of caution: this isomorphism does not hold when one takes cohomology with coefficients in $\Z$.  See \cite[Proposition~8.1]{bott} for a more precise statement.)
\end{remark}

\section{The equivariant correspondence}\label{s.eq-satake}

There is an equivariant version of Theorem~\ref{t.satake}, whose proof is sketched in \cite{ginzburg}.  We will write $t$ for the generic element of the Cartan subalgebra $\liet^\vee\subseteq \lieg^\vee$ and use the notation $\lieg^\vee[t]=\lieg^\vee\otimes\C[\liet^\vee]$ for the Lie algebra over the polynomial ring.  Given a $\lieg^\vee$-module $V$, there is an induced $\lieg^\vee[t]$-module $V[t] := V \otimes \C[\liet^\vee]$, where the action is given by
\[
 (x\otimes f)\cdot (v\otimes g) = (x\cdot v) \otimes (fg),
\]
for $x\in \lieg^\vee$, $v\in V$, and $f,g\in \C[\liet^\vee]$.

Next suppose $\lieb^\vee \subseteq \lieg^\vee$ is a Borel subalgebra containing $\liet^\vee$.  Any character $\chi$ of $\liet^\vee$ extends to one of $\lieb^\vee$, and also to $\lieb^\vee[t]$.  If $V$ is a $\lieb^\vee$-module, we can twist it by the character $\chi$ to obtain modules $V(\chi)$ and $V(\chi)[t]$ for $\lieb^\vee$ and $\lieb^\vee[t]$, respectively.  Concretely, if one writes an element of $\lieb^\vee = \lien^\vee \oplus \liet^\vee$ as $x=n+t$, then for $f,g\in \C[\liet^\vee]$ and $v\in V(\chi)$ a weight vector for $\liet^\vee$, we have
\begin{align*}
  (x\otimes f) \cdot (v\otimes g) &= (n\otimes f + t\otimes f)\cdot (v\otimes g) \\
  &= (n\cdot v)\otimes(fg) + (t\cdot v)\otimes(fg) + \chi(t\otimes f)(v\otimes g).
\end{align*}

Now let $\xi_t = \xi-t$ in $\lieg^\vee[t]=\lieg^\vee\otimes\C[\liet^\vee]$, where $\xi$ is a principal nilpotent as before, and $t$ is the generic element of the Cartan subalgebra $\liet^\vee\subseteq \lieg^\vee$.  Concretely, for $\liegl_n$ this is
\begin{equation}\label{e.xih}
 \xi_t = \left[\begin{array}{cccc}
-t_1 &  &  &  \\
1 & \ddots &  &  \\
0 & \ddots & \ddots &  \\
0 & 0 & 1 & -t_n\end{array}\right].
\end{equation}
Let $\liea_t\subseteq \lieg^\vee[t]$ be the centralizer of $\xi_t$.  For $\liegl_n$, this subalgebra is spanned over $\C[\liet^\vee] = \Sym^*\liet^\vee$ by the matrix powers $1, \xi_t, \xi_t^2,\ldots,\xi_t^{n-1}$.

\begin{theorem}[Equivariant Geometric Satake]\label{t.equiv-satake}
There are isomorphisms
\[
  H_{T}^*\GGr_{G}^\circ \isom \univ_{\C[\liet^\vee]}(\liea_t) = \Sym_{\C[\liet^\vee]}^*\liea_t
\]
inducing compatible actions on
\[
  IH^{T}_*(\overline{\GGr^\varpi}) \isom  V_\varpi(-\varpi)[t].
\]
\end{theorem}

The effect of twisting by the character $-\varpi$ is to move the highest weight vector of $V_\varpi$ to weight zero.  This corresponds to endowing $IH^{T}_*({\GGr^\varpi})$ with a $\liet^\vee$-module structure so that the fundamental class $[\overline{\GGr^\varpi}]$ has weight $0$.  This choice has the advantage of identifying the action of the element $\xi_t$ with equivariant multiplication by the divisor class $\sigma_{\bx}$.  (Of course, a similar isomorphism holds without the twist; noting that $\sigma_{\bx} = c_1^T( \OO(1)\otimes(-\varpi) )$, where $\OO(1)$ is the ample line bundle corresponding to the weight $\varpi$, the untwisted version identifies the action of $\xi_t$ with multiplication by $c_1^T( \OO(1) )$.)

\begin{example}
For $\lieg = \liegl_n$, the action of $\xi_t$ on $V_{\varpi_2}(-\varpi_2)[t] = (\exterior^2\C^n\otimes (-t_1-t_2))\otimes\C[\liet^\vee]$ is as follows.
\begin{align*}
  \xi_t\cdot (\e_1\wedge \e_2) &= \e_2\wedge\e_2 + \e_1\wedge\e_3 - (t_1+t_2-t_1-t_2) \e_1\wedge\e_2 \\
    &= \e_1\wedge\e_3,
\end{align*}
corresponding to $\sigma_{\bx}\cdot \sigma_\emptyset = \sigma_{(1)}$ in $H_{T}^*\Gr(2,n)$.  (The cancellation of the last term shows why the twist by $-\varpi$ is necessary.)  Similarly,
\begin{align*}
  \xi_t\cdot (\e_1\wedge \e_3) &= \e_2\wedge\e_3 + \e_1\wedge\e_4 - (t_1+t_3-t_1-t_2) \e_1\wedge\e_3 \\
    &= \e_2\wedge\e_3 + \e_1\wedge\e_4 + (t_2-t_3)\e_1\wedge\e_3,
\end{align*}
corresponding to $\sigma_{\bx}\cdot\sigma_{(1)} = \sigma_{(1,1)}+\sigma_{(2)} + (t_2-t_3)\sigma_{(1)}$.
\end{example}

As in the non-equivariant case, one needs to beware of the notation: matrix powers $\xi_t^j$ do not correspond to iterates of the Lie algebra action, e.g., $\xi_t^2\cdot \e_I$ is generally not equal to $\xi_t\cdot(\xi_t\cdot \e_I)$.

\begin{example}\label{ex.powers}
Still in the case $\lieg=\liegl_n$, let us consider higher powers of $\xi_t$.  One computes the entries of the matrix powers as
\begin{align*}
\xi_t^j \e_i &= \e_{i+j} - h_1(t_i,\ldots,t_{i+j-1})\,\e_{i+j-1} + \cdots + (-1)^j\,h_j(t_i)\,\e_i \\
    &= \sum_{a=0}^j (-1)^a\,h_a(t_i,\ldots,t_{i+j-a})\,\e_{i+j-a},
\end{align*}
where the $h_a$ are complete homogeneous symmetric polynomials in the indicated variables.  (That is, the $(i+j-a,i)$ matrix entry of $\xi_t^j$ is $(-1)^a h_a(t_i,\ldots,t_{i+j-a})$.)  Incorporating the twist by $-\varpi_k$, the Lie algebra action on $V_{\varpi_k}(-\varpi_k)[t]$ is
\begin{align*}
 \xi_t^j\cdot \e_I &= \left(\sum_{a=0}^j (-1)^a\,h_a(t_{i_1},\ldots,t_{i_1+j-a})\,\e_{i_1+j-a}\right) \wedge \e_{i_2} \wedge \cdots \wedge \e_{i_k} \\
    &\qquad + \e_{i_1} \wedge \left(\sum_{a=0}^j (-1)^a\,h_a(t_{i_2},\ldots,t_{i_2+j-a})\,\e_{i_2+j-a}\right) \wedge \cdots \wedge \e_{i_k} \\
    &\qquad + \cdots + \e_{i_1}\wedge\e_{i_2}\wedge\cdots\wedge\left(\sum_{a=0}^j (-1)^a\,h_a(t_{i_k},\ldots,t_{i_k+j-a})\,\e_{i_k+j-a}\right) \\
    & \qquad  - ( (-t_1)^j + \cdots + (-t_k)^j)\, \e_I.
\end{align*}
For instance,
\begin{align*}
 \xi_t^2\cdot (\e_2\wedge\e_3) &= (\e_4 - (t_2+t_3)\,\e_3 + t_2^2\,\e_2) \wedge \e_3 + \e_2\wedge ( \e_5 - (t_3+t_4)\,\e_4 + t_3^2\,\e_3 )\\
   & \qquad  - (t_1^2+t_2^2)\,\e_2\wedge\e_3 \\
   &= \e_2\wedge\e_5 - \e_3\wedge\e_4 - (t_3+t_4)\,\e_2\wedge\e_4 + (t_3^2-t_2^2)\,\e_2\wedge\e_3.
\end{align*}
The leading term agrees with the computation $p_2\cdot \sigma_{(1,1)} = \sigma_{(3,1)}-\sigma_{(2,2)}$ in $H^*\Gr(2,n)$ done in the introduction.
\end{example}

A new feature appears in the equivariant correspondence.  Let us pass to the fraction field $\C(\liet^\vee)$, and consider $\lieg^\vee(t)$, etc., as Lie algebras over this field.  Since the element $\xi_t$ is regular semisimple, its centralizer $\lieh\subseteq \lieg^\vee(t)$ is a Cartan subalgebra.  (In fact, $\lieh$ is just the extension of $\liea_t$ to $\C(\liet^\vee)$.)  So our setup leads naturally to another basis for $V_{\varpi}(-\varpi)(t)$, a basis of weight vectors for $\lieh(t)$, diagonalizing $\xi_t$.

What is this basis on the geometric side of the correspondence?  By the localization theorem (see \cite[(6.3)]{gkm}), there is a fixed-point basis for $IH^T_*({\GGr^\varpi})\otimes\C(\liet^\vee)$, and in fact this basis corresponds to a (suitably chosen) weight basis for $\lieh$.

\setcounter{theorem}{3}
\begin{theorem}[Equivariant Satake, continued]
Under the Satake isomorphism $IH^{T}_*(\overline{\GGr^\varpi}) \isom  V_\varpi(-\varpi)[t]$, equivariant $MV$-cycles correspond to a weight basis of $V_\varpi(-\varpi)[t]$ with respect to $\liet^\vee$, and the fixed point basis corresponds to a weight basis with respect to $\lieh$.
\end{theorem}
\setcounter{theorem}{6}

In general, there is ambiguity in choosing a weight basis.  However, for minuscule $\varpi$, all weight spaces of $V_\varpi$ are one-dimensional, so a weight basis is determined (up to scaling) by the Cartan.  As noted before, in this case $\GGr^\varpi = G/P_\varpi$ is homogeneous, and the MV basis consists of (opposite) Schubert classes.  This is the situation we will consider for the remainder of the paper.  Let us write $X=G/P_\varpi$.

Let us write $\{\sigma_\lambda\}$ for the basis of Schubert classes in $H_T^*X$.  The fixed point set is $X^T = \{p_\lambda\}$ (with $\lambda$ running over the same set indexing Schubert classes), and we will write $\{{\bf1}_\lambda\}$ for the corresponding idempotent basis of $H_T^*X^T = \bigoplus H_T^*(p_\lambda)$.  The localization theorem says that the restriction homomorphism (of $\C[\liet^\vee]$-algebras)
\[
 \iota^*\colon H_T^*X \to H_T^*X^T
\]
becomes an isomorphism after tensoring with $\C(\liet^\vee)$.  An important part of equivariant Schubert calculus is to compute the restriction of a Schubert class $\sigma_\lambda$ to a fixed point $p_\mu$.  Formulas for these restrictions have been given by Billey for complete flag varieties \cite{billey}, and by Ikeda-Naruse, who consider special cases that are related to the focus of this article \cite{in}.

The fact that the fixed-point classes form a basis of eigenvectors for $\xi_t$ is part of a general phenomenon, with a simple proof.  Consider any nonsingular variety $X$ with finite fixed locus $X^T$, and any class $\alpha\in (H_T^*X) \otimes \C(\liet^\vee)$.

\begin{lemma}\label{l.idempotent}
The idempotent classes ${\bf 1}_p \in (H_T^*X) \otimes \C(\liet^\vee) = (H_T^*X^T) \otimes \C(\liet^\vee)$ form a basis of eigenvectors for the endomorphism $x \mapsto \alpha\cdot x$.
\end{lemma}

\noindent
This is almost a tautology.  Simply observe that for distinct fixed points $p\neq q$, we have ${\bf 1}_p\cdot {\bf1}_q=0$.  Writing $\alpha = \sum \alpha_p\cdot {\bf1}_p$, the statement follows.

Returning to minuscule Schubert calculus, the restrictions $\sigma_\lambda|_\mu$ may be regarded as matrix entries for the homomorphism $\iota^*$, with respect to the Schubert and fixed-point bases.  The Satake correspondence translates the problem of computing this matrix into the following:

\medskip

\begin{quote}
 Find the change-of-basis matrix relating weight bases of the minuscule representation $V_{\varpi}(-\varpi)(t)$, with respect to two (specific) Cartan subalgebras, $\liet^\vee$ and $\lieh$, of $\lieg^\vee(t)$.
\end{quote}

\medskip

This perspective also suggests a framework for setting up and solving the problem of computing restrictions of Schubert classes $\sigma_\lambda|_\mu$.  We will work this out in types A and D in Sections~\ref{s.giambelli}--\ref{s.spinors}.

\begin{example}
Continuing our type A running example, consider the matrix expressing the Schubert basis $\sigma_i$ in terms of the fixed-point basis ${\bf1}_j$, for $\P^{3}$ (so $n=4$).  This is
\[
M=\left[\begin{array}{cccc}
1 & 0 & 0 & 0 \\
1 & t_1-t_2 & 0 & 0 \\
1 & t_1-t_3 & (t_1-t_3)(t_2-t_3) & 0 \\
1 & t_1-t_4 & (t_1-t_4)(t_2-t_4) & (t_1-t_4)(t_2-t_4)(t_3-t_4)
\end{array}\right].
\]
This represents the homomorphism $\iota^*\colon H_T^*\P^3 \to H_T^*(\P^3)^T$ with respect to the specified bases.

In our examples so far, we have written representations of $\liegl_4$ in terms of the Schubert basis, $\e_i=\sigma_i$.  By inverting the matrix $M$, we express the fixed point basis $f_j={\bf1}_j$ in terms of the $\e_i$'s.  One checks that this diagonalizes the regular semisimple operator $\xi_t$; that is, $M\cdot \xi_t\cdot M^{-1}$ is diagonal, with entries $-t_1,\ldots,-t_4$.  
So the $\e$ basis is a weight basis for the standard (diagonal) torus $\liet^\vee \subset \liegl_4$, while the $f$ basis is a weight basis for the centralizer $\lieh$ of $\xi_t$; the two are related by the matrix $M$ of the restriction homomorphism $\iota^*$.
%
\end{example}

\section{A Giambelli formula for Grassmannians}\label{s.giambelli}

We will describe a proof of the ``equivariant Giambelli formula''
\begin{equation}\label{e.giambelliA}
  \sigma_\lambda = s_\lambda(x|t)
\end{equation}
identifying the Schubert class $\sigma_\lambda \in H_T^*\Gr(k,n)$ with a {\it factorial Schur polynomial}, in the spirit of Laksov's approach to equivariant Schubert calculus \cite{laksov}.  The following definition of the factorial Schur polynomial can be found in Macdonald's book \cite[\S I.3, Ex.~20]{macdonald}.  The {\it (generalized) factorial power} is defined as
\[
  (x|t)^a = (x+t_1)(x+t_2)\cdots(x+t_a).
\]
Let $I = \{i_1,\ldots,i_k\} \subseteq [n]$ be the subset corresponding to the partition $\lambda$; recall that this means $\lambda_{k+1-a} = i_a-a$.  One defines
\begin{align*}
  s_\lambda(x|t) &=  \frac{\det\left( (x_j|t)^{i-1} \right)_{i\in I, 1\leq j\leq k}}{\det\left( (x_j|t)^{i-1} \right)_{1\leq i,j\leq k}}.
\end{align*}
An easy computation shows the denominator is
\[
 \det\left( (x_j|t)^{i-1} \right)_{1\leq i,j\leq k} = \det(x_j^{i-1}) =\prod_{1\leq a<b\leq k} (x_a-x_b) =: \Delta,
\]
the Vandermonde determinant, so the factorial Schur polynomial can also be written as
\[
 s_\lambda(x|t) = \frac{\det\left( (x_j|t)^{i-1} \right)_{i\in I, 1\leq j\leq k}}{\Delta}.
\]

The meaning of the Giambelli formula is this.  By the localization theorem, the equivariant cohomology of $\Gr(k,n)$ embeds in that of its fixed locus:
\[
  H_T^*\Gr(k,n) \hookrightarrow H_T^*(\Gr(k,n)^T) = \bigoplus_J \C[t],
\]
the sum being over all $k$-element subsets $J\subset [n]$.  On the other hand,
there is a presentation of $H_T^*\Gr(k,n)$ as a quotient of $\C[t][x_1,\ldots,x_k]^{S_k}$ (symmetric polynomials in $x$, with coefficients in $\C[t]$).  Composing with the localization homomorphism gives
\[
  \C[t][x_1,\ldots,x_k]^{S_k} \to \bigoplus_J \C[t],
\]
defined on the $J$th summand by sending $x_a \mapsto -t_{j_a}$.  The precise statement is this:

\begin{theorem}
Under the homomorphism $\C[t][x_1,\ldots,x_k]^{S_k} \to H_T^*\Gr(k,n)$, we have $s_\lambda(x|t) \mapsto \sigma_\lambda$.  Equivalently, for each $J=\{j_1<\cdots<j_k\}$, we have
\[
  \sigma_\lambda|_{p_J} = s_\lambda(-t_{j_1},\ldots,-t_{j_k}|t).
\]
\end{theorem}

\begin{proof}
There are three simple steps.  We describe them informally first, since we will follow the same pattern in proving a type D formula later.

\begin{enumerate}
\item Work out the case $k=1$, corresponding to projective space.  Here, by construction, the element $\xi_t$ corresponds to multiplication by the hyperplane class $\sigma_{\bx}$ on $H_T^*\P^{n-1}$, written in the Schubert basis $\e_i = [\P^{n-i}]$.  We choose a basis $\barf_i$ diagonalizing the semisimple element $\xi_t$; by Lemma~\ref{l.idempotent}, this basis coincides with the basis of idempotents ${\bf1}_i$, up to scalar.  We normalize the $\barf_i$ so that $\barf_i={\bf1}_i$, by requiring $\e_1 = \barf_1 + \cdots + \barf_n$ (since $\e_1$ corresponds to ${\bf1}\in H_T^*\P^{n-1}$).  The expansion of $\e_i$ in the $\barf_j$ basis is then a localization calculation, which is easy for projective space.

\medskip

\item For each $k>1$, take the weight basis $\{\e_I\}$ of $V_{\varpi_k}(-\varpi_k)[t] = \exterior^k_{\C[t]} \C[t]^n$ to be $\e_I = \e_{i_1}\wedge \cdots \wedge \e_{i_k}$.  Verify that the action of $\xi_t$ agrees with the known formula for multiplication by $\sigma_{\bx}$ on the Schubert basis of $H_T^*\Gr(k,n)$, so that we can identify $\sigma_I = \e_I$.  (The latter formula is often called the {\it equivariant Chevalley formula}.)

\medskip
\item By Lemma~\ref{l.idempotent} again, the vectors $\barf_{j_1}\wedge \cdots \wedge \barf_{i_k}$ agree with the basis of idempotents ${\bf1}_J$, up to scalar; normalize it so that $\barf_J={\bf1}_J$ by requiring $\e_{\{1,\ldots,k\}} = \sum \barf_J$ (since $\e_{\{1,\ldots,k\}}$ corresponds to ${\bf1} \in H_T^*\Gr(k,n)$).  On the other hand, formulas from Step (1) expressing $\e_i$ in terms of $\barf_j$ yield (determinantal) formulas for $\e_I$ in terms of $\barf_{j_1}\wedge \cdots \wedge \barf_{j_k}$; comparing with the normalized vectors $\barf_J$ proves the theorem.
\end{enumerate}

Now we proceed to work this out in detail.  It is not hard to see that
\[
  f_i = \e_i + \frac{1}{(t_{i+1}-t_i)}\e_{i+1} + \cdots + \frac{1}{(t_n-t_i)\cdots(t_{i+1}-t_i)}\e_n
\]
is a basis of eigenvectors for $\xi_t$ acting on $\C^n\otimes \C(t)$ (inverting nonzero characters).  This is related to $\e_i$ by a unitriangular change of basis. However, note that $\sigma_i|_{p_i} = (t_1-t_i)\cdots (t_{i-1}-t_i)$ (since $\Omega_i = \P^{n-i}$ is defined by the vanishing of the first $i-1$ coordinates); this means that we must rescale to obtain the idempotent basis.  In fact,
\[
  \barf_i = \frac{1}{(t_1-t_i)\cdots (t_{i-1}-t_i)}f_i
\]
identifies with the idempotent basis ${\bf1}_i\in H_T^*(\P^{n-1})^T$.  
Since we know the restrictions of $\e_i = [\P^{n-i}]$ to fixed points, we see
\begin{align*}
  \e_i &= \sum_j (t_1-t_j)\cdots(t_{i-1}-t_j)\,\barf_j \\
       &= \sum_j (x_j|t)^{i-1}|_{x_j=-t_j}\, \barf_j.
\end{align*}
using the generalized factorial power notation.  This completes the first step.

For the second step, we take $\e_I = \e_{i_1}\wedge \cdots \wedge \e_{i_k}$ as our basis for $V_{\varpi_k}(-\varpi_k)[t] = H_T^*\Gr(k,n)$.  The verification that $\xi_t$ acts on this basis as $\sigma_{\bx}$ does on the Schubert basis is left to the reader. (Illustrative examples were done above.)  We note that the Chevalley formula says
\[
  \sigma_{\bx}\cdot \sigma_I = \sum_{I^+} \sigma_{I^+} + (t_1+\cdots+t_k-t_{i_1}-\cdots-t_{i_k})\sigma_I,
\]
where the sum is over $I^+$ obtained from $I$ by replacing some $i_a\in I$ such that $i_a+1\not\in I$ by $i_a+1$.  (In terms of the corresponding partitions, $\lambda(I^+)$ is obtained from $\lambda(I)$ by adding a single box.)

Finally, from the definition of exterior product, we get
\[
  \e_I := \e_{i_1}\wedge\cdots\wedge\e_{i_k} = \sum_J \det\left( (x_j|t)^{i-1}|_{x_j=-t_j} \right)_{i\in I, \,j\in J} \, \barf_{j_1}\wedge\cdots\wedge\barf_{j_k}.
\]
In particular,
\begin{align*}
  \e_{\{1,\ldots,k\} } &= \sum_J \det\left( (x_j|t)^{i-1}|_{x_j=-t_j} \right)_{1\leq i\leq k, \,j\in J} \, \barf_{j_1}\wedge\cdots\wedge\barf_{j_k} \\
    &= \sum_J \Delta_J \, \barf_{j_1}\wedge\cdots\wedge\barf_{j_k},
\end{align*}
where
\[
  \Delta_J = \prod_{1\leq a<b\leq k} (t_{j_a}-t_{j_b})
\]
is the specialization of the Vandermonde determinant $\Delta$.  Since $\e_{\{1,\ldots,k\} }$ should be identified with ${\bf1}=[\Gr(k,n)]$ in $H_T^*\Gr(k,n)$, this tells us that the idempotent classes are
\[
  {\bf1}_J = \barf_J := \Delta_J \,\barf_{j_1}\wedge\cdots\wedge\barf_{j_k},
\]
and we can rewrite the above formula as
\begin{align*}
  \e_I := \e_{i_1}\wedge\cdots\wedge\e_{i_k} &= \sum_J \left( \frac{\det\left( (x_j|t)^{i-1}|_{x_j=-t_j} \right)_{i\in I, \,j\in J}}{\Delta_J} \right)\, \barf_J \\
    &= \sum_J s_{\lambda}(x_1,\ldots,x_k|t)|_{x_a = -t_{j_a}} \barf_J,
\end{align*}
as required.
\end{proof}

\begin{remark}
It is hard to give clear attribution to the equivariant Giambelli formula; certainly it was known by around 2000.  Reference to it appears in \cite{knutson-tao}, and a proof is in \cite{mihalcea-giambelli}.  In retrospect, the Kempf-Laksov formula \cite{kl} is equivalent to \eqref{e.giambelliA}.  See also \cite{poland} for more discussion and an alternative proof.

Likewise, it is difficult to identify the earliest appearance of the connection between the $\liegl_n$-module $\exterior^k\C^n$ and the cohomology of $H^*\Gr(k,n)$.  While surely known long before, it appears in several sources by the 2000s \cite{gatto,laksov,varagnolo,kamnitzer}.
\end{remark}

\begin{remark}
Recall that the Grassmannian $\Gr(k,n) = \Gr(k,V^*)$ embeds naturally in $\P(\exterior^kV^*)$ as the locus of ``pure wedges'' $v_1 \wedge \cdots \wedge v_k$.  Since our basis $\e_1,\ldots,\e_n$ for $V$ is a weight basis for $\liet^\vee$, the dual basis $\e^*_1,\ldots,\e^*_n$ is a weight basis for the action of $T$ on $V^*$; the points of $\Gr(k,V^*)$ corresponding to $\e^*_I = \e^*_{i_1}\wedge\cdots\wedge \e^*_{i_k}$ are therefore precisely the $T$-fixed points.  Thus the Satake correspondence exchanges $T$-fixed points in $\Gr(k,V^*)\subseteq\P(\exterior^kV^*)$ with Schubert classes in $\P(H^*\Gr(k,V)) = \P(\exterior^kV)$.  When considered as cohomology classes on $\Gr(k,V^*)$, do pure wedges in $\P(\exterior^kV)$ have a natural geometric meaning?  What does this correspondence look like when upgraded to the equivariant setting?
\end{remark}

\section{Quadrics}\label{s.quadrics}

As another example, we describe the correspondence for minuscule varieties of type $D$.  First we consider quadrics.  The calculations carried out in this section will be used in the section, to prove the equivariant Giambelli formula for maximal isotropic Grassmannians (Theorem~\ref{t.giambelliD}).

To set things up, fix a basis
\[
  \e_{\bar{n-1}},\ldots,\e_{\bar{1}},\e_{\bar{0}},\e_0,\e_1,\ldots,\e_{n-1}
\]
for $V\isom\C^{2n}$, and equip this vector space with the symmetric bilinear form defined by $\langle \e_{\bar\imath},\,\e_j \rangle = \delta_{i,j}$.  (The barred indices should be regarded as notation for negative integers.)  The form identifies the dual basis for $V^*=V$ as $\e_i^*=\e_{\bar\imath}$.  The subspaces $E=\Span\{\e_0,\ldots,\e_{n-1}\}$ and $\bar{E} = \Span\{\e_{\bar{n-1}},\ldots,\e_{\bar{0}} \}$ are maximal isotropic subspaces of $V$, and we have $V=\bar{E}\oplus E$.

We will take $\lieso_{2n} \subseteq \liesl_{2n}$ to be the algebra preserving the given bilinear form; there is also a canonical identification $\lieso_{2n} = \exterior^2V$ (see, e.g., \cite[\S20]{fh}).  We take our principal nilpotent element $\xi$ and generic $t$ so that
\begin{equation}\label{e.d-xit}
\xi_{t} = \xi-t = \left[\begin{array}{cccccccc}
t_{n-1} &   &   &   &  &   &  &   \\
-1 & \ddots &   &   &   &   &   &   \\
 & \ddots & t_1 &   &   &   &   &   \\
 &   & -1 & t_0 & 0 &   &   &   \\
 &   & -1 & 0 & -t_0 &   &   &   \\
 &   &   & 1 & 1 & \ddots &   &   \\
 &   &   &   &   & \ddots & -t_{n-2} &   \\
 &   &   &   &   &   & 1 & -t_{n-1}\end{array}\right] .
\end{equation}
Using $\e_{\bar\imath}^*=\e_i$, $\xi$ can also be written as
\begin{align*}
  \xi &= -\sum_{i=1}^{n-1} \e_{\bar{\imath}}^*\otimes \e_{\bar{i-1}} - \e_{\bar{1}}^*\otimes \e_0 + \e_{\bar{0}}^*\otimes \e_1  + \sum_{i=1}^{n-1} \e_{i-1}^*\otimes \e_{i} \\
  &= \e_0 \wedge \e_1 + \sum_{i=1}^{n-1} \e_{\bar{i-1}} \wedge \e_{i} ,
\end{align*}
which exhibits it as an element of $\exterior^2V$.  Similarly, we have
\begin{align}
 \xi_t &= \e_0 \wedge \e_1 + \sum_{i=1}^{n-1} \e_{\bar{i-1}} \wedge \e_{i} - \sum_{i=0}^{n-1} t_i\, \e_{\bar\imath} \wedge \e_i.
\end{align}
Note that our indexing conventions and choice of form $\langle\; ,\; \rangle$ make it natural to identify elements of $\exterior^2 V$ with matrices which are skew-symmetric about the {\it anti-diagonal}.

The odd matrix powers $\xi_t, \xi_t^3, \ldots, \xi_t^{2n-3}$ all lie in $\lieso_{2n}$ as well, and they are easily seen to be linearly independent elements of the centralizer $\liea_t$ of $\xi_t$.  Since $\xi_t$ is regular, one knows $\dim\liea_t = n$; the missing element is
\begin{equation}
  \eta_t =  -\sum_{j=1}^{n-1} (t_{j+1}\cdots t_{n-1}) \e_0\wedge \e_j + \sum_{0\leq i\leq j \leq n-1} (t_0 \cdots t_{i-1} t_{j+1} \cdots t_{n-1}) \e_{\bar\imath} \wedge \e_j.
\end{equation}
For example, when $n=4$, this is
\begin{equation*}\label{e.d-etat}
\eta_{t} = \left[\begin{array}{cccccccc}
-t_0 t_1 t_2 &   &   &   &  &   &  &   \\
-t_0 t_1   & -t_0 t_1 t_3  &   &   &   &   &   &   \\
-t_0     &  -t_0 t_3      & -t_0 t_2 t_3 &   &   &   &   &   \\
-1      &  -t_3 & -t_2 t_3 & -t_1 t_2 t_3 &  &   &   &   \\
 1      & t_3   & t_2 t_3  &   0 & t_1 t_2 t_3 &   &   &   \\
 0      & 0     &  0       & -t_2 t_3 & t_2 t_3 & t_0 t_2 t_3  &   \\
 0      & 0     &  0       & -t_3    & t_3  & t_0 t_3  & t_0 t_1 t_3 &   \\
 0      & 0     &  0       &  -1     &  1   &  t_0     & t_0 t_1   & t_0 t_1 t_2 \end{array}\right] .
\end{equation*}

When discussing homogeneous spaces, we will assume $n\geq 3$ to avoid setting conventions for special cases.\footnote{When $n=1$, the spaces are $0$-dimensional; when $n=2$, they coincide with type $A$ spaces.  For $n=3$, there are  coincidences $\QQ^4 = \Gr(2,4)$ and $OG^+(3,6)=OG^-(3,6)=\P^3$.  For $n=4$, there are also coincidences $\QQ^6 = OG^+(4,8) = OG^-(4,8)$.  The reader may use these to verify our claims, but beware that the torus actions are usually written differently.}
Consider the $(2n-2)$-dimensional quadric $\QQ = \QQ^{2n-2} \subseteq \P(V^*)\isom  \P^{2n-1}$ of isotropic vectors for the given bilinear form.  (In coordinates, $\QQ$ is defined by the vanishing of the quadratic form $\sum_{i=0}^{n-1} X_{\bar{\imath}}\, X_i$, where $X_i = \e_i^*$.)  The torus $T\isom(\C^*)^n$ acts on $V$ with weights $-t_{n-1},\ldots,-t_0,t_0,\ldots,t_{n-1}$, inducing an action on $\QQ$.

The quadric $\QQ$ is homogeneous for $PSO_{2n}$, and the Satake correspondence identifies $H_T^*\QQ$ with $V_{\varpi_1}(-\varpi_1)[t]$, where $\varpi_1 = -t_{n-1}$ and $V_{\varpi_1}=V$ is the standard representation of $\lieso_{2n}$.  To see this explicitly, define Schubert varieties in $\QQ$ by
\begin{align*}
 \Omega_{\bar\imath} &= \{ \e^*_{\bar{n-1}} = \cdots = \e^*_{\bar{i+1}} = 0 \}, \\
 \Omega_{i} &= \{ \e^*_{\bar{n-1}} = \cdots = \e^*_{\bar{0}} = \e^*_0 = \cdots = \e^*_{i-1} = 0  \}
\end{align*}
for $i>0$; and
\begin{align*}
 \Omega_{\bar{0}} &= \{ \e^*_{\bar{n-1}} = \cdots = \e^*_{\bar{1}} = \e^*_{\bar{0}} = 0 \} \\
  \Omega_{0} &= \{ \e^*_{\bar{n-1}} = \cdots = \e^*_{\bar{1}} = \e^*_0 = 0  \}.
\end{align*}
Identify the Schubert classes $\sigma_i = [\Omega_i]$ in $H^T_{n-1-i}\QQ = H_T^{n-1+i}\QQ$ with basis elements $\e_i$ by
\begin{align}\label{e.quadric-schub}
  \sigma_{\bar\imath} = (-1)^i\, \e_{\bar\imath} \quad \text{and} \quad   \sigma_{i}  = \e_i,
\end{align}
for $i\geq 0$.
Using the twist by ${t_{n-1}}$, we have
\begin{align*}
  \xi_t\cdot \e_{\bar\imath} &= -\e_{\bar\imath+1} + (t_i-t_{n-1})\e_{\bar\imath} \qquad \text{for } i>1, \\
  \xi_t\cdot \e_{\bar{1}} &=  -\e_{\bar{0}} -\e_0 + (t_1-t_{n-1})\e_{\bar{1}} ,\\
  \xi_t\cdot \e_{\bar{0}} &= \e_1 + (t_0-t_{n-1})\e_{\bar{0}}, \quad\text{and} \\
  \xi_t\cdot \e_i &=  \e_{i+1} + (-t_i-t_{n-1})\e_i \qquad \text{for }i\geq 0.
\end{align*}
On the other hand, taking $\sigma_{\bx} = \sigma_{\bar{n-2}}\in H_T^1\QQ$ to be the hyperplane class,
\begin{align*}
  \sigma_{\bx}\cdot \sigma_{\bar\imath} &= \sigma_{\bar\imath+1} + (t_i-t_{n-1})\sigma_{\bar\imath} \qquad \text{for } i>1, \\
  \sigma_{\bx}\cdot \sigma_{\bar{1}} &=  \sigma_{\bar{0}} +\sigma_0 + (t_1-t_{n-1})\sigma_{\bar{1}} ,\\
  \sigma_{\bx}\cdot \sigma_{\bar{0}} &= \sigma_1 + (t_0-t_{n-1})\sigma_{\bar{0}}, \quad\text{and} \\
  \sigma_{\bx}\cdot \sigma_i &=  \sigma_{i+1} + (-t_i-t_{n-1})\sigma_i \qquad \text{for }i\geq 0,
\end{align*}
so \eqref{e.quadric-schub} compatibly identifies the action of $\xi_t$ with the product by $\sigma_{\bx}$.

Iterating the above computation of $\xi_t\cdot\e_i$ leads to a formula for the matrix entries of $\xi_t^{2j-1}$, for $j=1,\ldots,n-1$:
\begin{equation}\label{e.xit-power}
\begin{aligned}
  \xi_t^{2j-1} &= \sum_{k=0}^{n-1}\sum_{i=0}^{n-1-k} (-1)^{2j-1-k} h_{2j-1-k}(t_i,\ldots,t_{i+k})\, \e_{\bar\imath}\wedge\e_{i+k} \\
    &\quad + \sum_{k=1}^{n-1}(-1)^{2j-1-k} h_{2j-1-k}(-t_0,t_1,\ldots,t_{k})\, \e_0\wedge \e_{k} \\
    & \quad + 2 \sum_{i=1}^{n-2} \sum_{k=i+1}^{n-1} (-1)^{2j-1-k} h_{2j-1-i-k}(-t_{i},-t_{i-1},\ldots,-t_0,t_0,\ldots,t_{k}) \,\e_i\wedge \e_{k},
\end{aligned}
\end{equation}
where in the last sum, the complete homogeneous symmetric functions are in $i+k+2$ variables, specialized as indicated to consecutive $t$'s.

The fixed points in the quadric $\QQ$ are the $2n$ coordinate points:
\[
\QQ^T = \{p_{\bar{n-1}},\ldots,p_{\bar{0}}, p_0,\ldots,p_{n-1}.\}
\]
That is, $p_j\in \QQ \subseteq \P^{2n-1}$ is the point with $1$ in the $j$th coordinate and $0$ elsewhere.  For each $i$, one computes the weights of $T$ acting on the tangent space $T_{p_i}\QQ$ to be $\{ t_j - t_i \,|\, j \neq i, \bar\imath\}$, using the notation $t_{\bar\jmath} = -t_j$.

Using the defining equations, it is easy to write down the restrictions of Schubert classes to fixed points.  For $i>0$, and any $j$, we have
\begin{align*}
  \sigma_{\bar\imath}|_{p_j} &= (-t_{n-1}-t_j)\cdots(-t_{i+1}-t_j) \\
                &= \prod_{k=i+1}^{n-1} (-t_k-t_j),
\end{align*}
Thus $\sigma_{\bar\imath}|_{p_j} = 0$ for $j<\bar\imath$, since the factor $(-t_{\bar\jmath}-t_j)$ is zero---and indeed, in this case $p_j \not\in\Omega_{\bar\imath}$.

For $i=\bar{0}$, we have
\begin{align*}
 \sigma_{\bar{0}}|_{p_{\bar{0}}} &= \prod_{k=1}^{n-1}(-t_k+t_0) ;\\
 \sigma_{\bar{0}}|_{p_j} &= (t_0-t_j)\mathop{\prod_{k=0}^{n-1}}_{k\neq j}(-t_k-t_j) \quad \text{ for }j>0;
\end{align*}
and all other restrictions are zero.  Finally, for $i\geq 0$,
\begin{align*}
 \sigma_i|_{p_j} &= \left(\mathop{\prod_{k=0}^{n-1}}_{k\neq j}(-t_k-t_j)\right)\prod_{k=0}^{i-1}(t_k-t_j) \quad \text{ for }j\geq i,
\end{align*}
and all other restrictions are zero.

Finally, we translate these calculations into a change of bases for the representation $V$.  Using $\e_{\bar\imath} = (-1)^i\sigma_{\bar\imath}$ and $\e_i = \sigma_i$ for $i\geq0$, we have $[\QQ] = \sigma_{\bar{n-1}} = (-1)^{n-1}\e_{\bar{n-1}}$, so by setting $\barf_i = (-1)^{n-1} {\bf1}_i$, where ${\bf1}_i \in H_T^*\QQ^T$ is the idempotent class at $p_i$, we have
\[
  \e_{\bar{n-1}} = \barf_{\bar{n-1}} + \cdots + \barf_{\bar0} + \barf_0 + \cdots + \barf_{n-1}.
\]
More generally, for $i>0$ we have
\[
 \e_{\bar\imath} = \sum_{j=\bar\imath}^{n-1}\left( \prod_{k=i+1}^{n-1} (t_k+t_j) \right)\barf_j;
\]
for $i=\bar{0}$ we have
\[
 \e_{\bar0} = \left(\prod_{k=1}^{n-1}(t_k-t_0)\right) \barf_{\bar0} + \sum_{j=1}^{n-1}\left((t_j-t_0) \prod_{k>0, k\neq j}(t_k+t_j) \right)\barf_j;
\]
and for $i\geq 0$ we have
\[
  \e_i = \sum_{j=i}^{n-1} \left( \prod_{k=0}^{i-1} (t_k^2-t_j^2)  \mathop{\prod_{k=i}^{n-1}}_{k\neq j} (t_k+t_j) \right) \barf_j.
\]

It will be convenient to rescale the basis $\{\barf_i\}$ so that $\{\e_i\}$ is related by a unitriangular change of basis.  To this end, for each $i$ let
\[
  f_i = \alpha_i \barf_i,
\]
where, for $i\geq 0$, the scaling coefficients are $\alpha_{\bar\imath} = (-1)^{n-1-i}\sigma_{\bar\imath}|_{p_{\bar\imath}}$ and $\alpha_i = (-1)^{n-1}\sigma_i|_{p_i}$.  Now we may write, for $i\geq 0$,
\[
 \e_{\bar\imath} = \sum_{j=0}^i \bar{c}_{ji} f_{\bar\jmath} + \sum_{j=0}^{n-1} c_{ji} f_j,
\]
and
\[
 \e_i = \sum_{j=0}^i \bar{b}_{ji} f_j.
\]
Explicitly, the matrices $\bar{C} = (\bar{c}_{ji})$, $C = (c_{ji})$, and $\bar{B}=(\bar{b}_{ji})$ are computed as follows.  For $i>0$,
\begin{align*}
 \bar{c}_{ji} &= \frac{1}{\prod_{k=j+1}^i(t_k-t_j)} \quad \text{ for }0\leq j\leq i; \\
 c_{ji} &= \frac{1}{\prod_{k=0}^{j-1}(t_k^2-t_j^2)\prod_{k=j+1}^{i-1}(t_k+t_j)} \quad \text{ for }0\leq j\leq i;
 \intertext{and}
 c_{ji} &= \frac{2t_j}{\prod_{k=0}^i(t_k^2-t_j^2)\prod_{k=i+1}^{j-1}(t_k-t_j)} \quad \text{ for } j>i.
 \end{align*}
For $i=0$,
\begin{align*}
  \bar{c}_{00} &=1 ; \\
  c_{00} & = 0; \\
  \intertext{and}
  c_{j0} &= \frac{1}{(-t_0-t_j)\prod_{k=1}^{j-1}(t_k-t_j)}.
\end{align*}
Finally, for $i\geq 0$,
\begin{align*}
  \bar{b}_{ji} &= \frac{1}{\prod_{k=i}^{j-1}(t_k-t_j)}.
\end{align*}
Note the matrices $\bar{C}$ and $\bar{B}$ are indeed unitriangular.

\section{Orthogonal Grassmannians}\label{s.spinors}

Now we turn to the maximal orthogonal Grassmannians $OG^\pm(n,2n)$, also known as {\em spinor varieties}.  We will maintain the notation from the previous section, so $V=\C^{2n}$ has a bilinear form and basis $\e_i$ so that $\langle \e_{\bar\imath}, \e_j \rangle = \delta_{ij}$.  As noted above, the subspace $E\subseteq V$ spanned by 
$\e_0,\e_1,\ldots,\e_{n-1}$
is isotropic with respect to the bilinear form, as is the complementary subspace $\bar{E} = \Span\{\e_{\bar{0}},\ldots,\e_{\bar{n-1}}\}$.  The orthogonal Grassmannian $OG^+(n,2n)$ (respectively, $OG^-(n,2n)$) parametrizes all $n$-dimensional isotropic subspaces $L\subseteq V$ such that $\dim(E\cap L)$ is even (resp., odd).  We will focus on the ``$+$'' case, and write $OG(n) = OG^+(n,2n)$ from now on.

The torus $T=(\C^*)^n$ acts on $OG(n)$ via its action on $V=\C^{2n}$; recall that this is given by weights $-t_{n-1},\ldots,-t_0,t_0,\ldots,t_{n-1}$.  The $T$-fixed points in $OG(n)$ are indexed by subsets $I \subseteq \{0,\ldots,n-1\}$ such that the cardinality of $I$ is even.  For such a subset, the fixed point $p_I$ corresponds to the subspace
\[
 E_I = \Span\left( \{ \e_i \,|\, i\in I \} \cup \{ \e_{\bar\jmath} \,|\, j\not\in I\}\right).
\]
For example, $p_{\emptyset} = \bar{E}$.

Schubert varieties in $OG(n)$ are also indexed by subsets $I \subseteq \{0,\ldots,n-1\}$ of even cardinality.  As for the ordinary Grassmannian, the elements of $I$ index pivots for Schubert cells $\Omega_I^\circ$, and $\Omega_I$ is the closure.  The isotropicity conditions mean that exactly one of $i$ or $\bar\imath$ occurs as a pivot, for $0\leq i\leq n-1$, and we record the positive ones.  For example, in $OG(4)$ we have
\[
  \Omega^\circ_{\{1,3\}} = \left[\begin{array}{cccccccc}
  0 & 1 & * & 0 &   *     & 0 & \bullet & 0 \\
  0 & 0 & 0 & 1 & \bullet & 0 & \bullet & 0 \\
  0 & 0 & 0 & 0 &    0    & 1 & \bullet       & 0 \\
  0 & 0 & 0 & 0 &    0    & 0 & 0       & 1 \end{array}\right].
\]
(From left to right, the columns are numbered $\bar{3},\bar{2},\bar{1},\bar{0},0,1,2,3$.  Stars are free entries, and bullets indicate entries that are dependent on the others, by the isotropicity condition.)  Similarly, $\Omega_\emptyset = OG(n)$, and $\Omega_{\{(0),1,\ldots,n-1\}} = \{p_{\{(0),1,\ldots,n-1\}}\}$, where $0$ is included or not, depending on the parity of $n$.  Schubert varieties are $T$-invariant.

Frequently one interprets the subsets $I$ as {\it strict partitions} $\lambda$, simply by reversing order from increasing to decreasing.  In this context, we will usually prefer the subset notation to partition notation, although the latter is useful for indicating containment relations: if $I$ and $J$ are subsets corresponding to partitions $\lambda$ and $\mu$, respectively, then $\Omega_I \subseteq \Omega_J$ if and only if $\lambda\supseteq\mu$ as Young diagrams.  We will write $I\geq J$ in this case.  (If $I=\{i_1<\cdots<i_r\}$ and $J=\{j_1<\cdots<j_s\}$, then $I\geq J$ is equivalent to $r\geq s$ and $i_a\geq j_a$ for $1\leq a\leq s$.)  
We write $\sigma_I = [\Omega_I]$ for the equivariant class of a Schubert variety; it has degree $|I| := \sum i_a$.

Our main goal in this section is to compute formulas for the restrictions $\sigma_I|_{p_J}$. 
Since $p_I$ is the unique fixed point in the Schubert cell $\Omega_I^\circ$, we have $p_I \in \Omega_J$ if and only if $I\geq J$.  From matrix representatives, it is easy to see that the normal space to $\Omega_I^\circ \subset OG(n)$ at the point $p_I$ has weights $\{-t_i+t_j\,|\, i>j;\, i\in I, j\not\in I \} \cup \{ -t_i-t_j \,|\, i>j;\, i,j\in I\}$.  It follows that
\begin{equation}\label{e.OGrestrict}
  \sigma_I|_{p_I} = \prod_{i\in I}\left( \mathop{\prod_{j\not\in I}}_{j<i}(-t_i+t_j) \mathop{\prod_{j\in I}}_{j<i}(-t_i-t_j) \right).
\end{equation}

The corresponding minuscule representation is the {\it half-spin representation} $\spin^+$ of $\lieso_{2n}$.  A brief description, suitable for our purposes, is in the appendix; to see this worked out in detail, we recommend \cite[\S20]{fh}, \cite{chevalley}, or \cite{manivel}.

Recall our standard representation $V$ of $\lieso_{2n}$ splits into maximal isotropic subspaces $V=\bar{E}\oplus E$, and we have fixed a basis $\e_i$ so that $\e_{\bar{n-1}},\ldots,\e_{\bar0}$ span $\bar{E}$, and $e_0,\ldots,e_{n-1}$ span $E$.  As noted in the appendix, $\spin=\spin^+\oplus\spin^-$ is an ideal of the Clifford algebra $Cl(V)$, and $\spin^+$ has a basis of elements
\begin{equation}\label{e.spin-basis-e}
  \e_I := \e_{\bar\imath'_1}\cdots \e_{\bar\imath'_{n-r}} \cdot \e,
\end{equation}
for $I = \{i_1 <\cdots < i_r\}\subseteq \{0,\ldots,n-1\}$ of even cardinality, with complement $I' = \{i'_1 < \cdots < i'_{n-r} \}$.  Here $\e = \e_0\cdots \e_{n-1}$.

The Cartan subalgebra is spanned by $\e_{\bar\imath}\wedge \e_i$, and its eigenvalue on the weight vector $\e_I$ is computed to be
\[
  \frac{1}{2}\left(\sum_{j\not\in I}t_j - \sum_{i\in I} t_i \right).
\]
In particular, the highest weight vector $\e_\emptyset$ has weight $\varpi_n = \frac{1}{2}\sum_{j} t_j$.  Twisting by ${-\varpi_n}$, the action of $\liet^\vee$ on $V_{\varpi_n}(-\varpi_n)[t]$ has $t\cdot \e_I = (-\sum_{i\in I} t_i)\,\e_I$.  Straightforward computations also show
\begin{align*}
  \e_{\bar\imath} \wedge \e_{i+1} \cdot \e_I &= \e_{\bar\imath}\cdot \e_{i+1}\cdot \e_I \\
    &= \begin{cases} \e_{I^+} & \text{if } i\in I \text{ and }i+1\not\in I; \\ 0 & \text{otherwise;} \end{cases}
\end{align*}
here $I^+ = (I \setminus \{i\}) \cup \{i+1\}$.  Similarly, $\e_0 \wedge \e_1 \cdot \e_I = \e_{I \cup \{0,1\}}$ if $0,1\not\in I$, and is zero otherwise.

This is enough to compute the action of $\xi_t$ on $\spin^+$.  For example,
\begin{align*}
 \xi_t \cdot \e_{\emptyset} &= \e_{\{0,1\}}; \\
 \xi_t \cdot \e_{\{0,1\}} & = (\e_{\bar{1}}\wedge\e_2)\cdot \e_{\{0,1\}} - t\cdot \e_{\{0,1\}}  \\
                &= \e_{\{0,2\}} + (t_0+t_1)\,\e_{\{0,1\}}; \\
 \xi_t \cdot \e_{\{0,2\}} & = (\e_{\bar{0}}\wedge\e_1+\e_{\bar{2}}\wedge\e_3)\cdot \e_{\{0,2\}} - t\cdot \e_{\{0,2\}}  \\
               &= \e_{\{1,2\}} +\e_{\{0,3\}} + (t_0 + t_2)\, \e_{\{0,2\}}.
\end{align*}
The Schubert classes are identified by
\[
  \sigma_I = \e_I,
\]
so writing $\sigma_{\bx} = \sigma_{\{0,1\}}$ for the divisor class, the above calculation agrees with
\begin{align*}
 \sigma_{\bx} \cdot 1 &= \sigma_{\{0,1\}}; \\
 \sigma_{\bx} \cdot \sigma_{\{0,1\}} & = \sigma_{\{0,2\}} + (t_0+t_1)\,\sigma_{\{0,1\}}; \\
 \sigma_{\bx} \cdot \sigma_{\{0,2\}} & = \sigma_{\{1,2\}} +\sigma_{\{0,3\}} + (t_0 + t_2)\, \sigma_{\{0,2\}},
\end{align*}
and once again the action of $\xi_t$ corresponds to multiplication by $\sigma_{\bx}$.

Now we are ready for the equivariant Giambelli formula for orthogonal Grassmannians.  This says
\begin{align}\label{e.giambelliD}
  \sigma_I = P_\lambda(x|t),
\end{align}
where $\lambda$ is the strict partition corresponding to $I$ (i.e., write $I$ in decreasing order), and $P_\lambda(x|t)$ is the {\em factorial Schur P-function}.  These polynomials were studied by Ivanov \cite{ivanov}; setting $t=0$, they specialize to Schur's $P$-functions (see \cite[III.8]{macdonald}), which were shown to represent Schubert classes by Pragacz \cite{pragacz}.  The equivariant formula \eqref{e.giambelliD} was proved by Ikeda and Naruse \cite{in}.  We will see a new proof of this formula.

Ivanov gives a formula for $P_\lambda(x|t)$ as a ratio of Pfaffians, similar to the one defining $s_\lambda(x|t)$, and inspired by Nimmo's formula for $P_\lambda(x)=P_\lambda(x|0)$ \cite{nimmo}.  Let $A(x) = (a_{ij}(x))$ be the $n\times n$ skew-symmetric matrix with
\[
  a_{ij}(x) = \frac{x_i-x_j}{x_i+x_j},
\]
for $0\leq i,j \leq n-1$.  Given $I=\{i_1<\cdots<i_r\}$, let $B_I(x|t) = (b_{kl}(x|t))$ be the $n\times r$ matrix with
\[
  b_{kl}(x|t) = (x_k|t)^{i_l}
\]
Form the skew-symmetric matrix\footnote{The indexing most natural to our setup is slightly nonstandard.  The rows and columns of $A(x)$ are labelled $n-1,\ldots,0$ (left to right, top to bottom); similarly, the rows of $B_I(x|t)$ are labelled $n-1,\ldots,0$ (top to bottom) and its columns are labelled $i_1,\ldots,i_k$ (left to right).}
\[
  A_I(x|t) =  \left[\begin{array}{c|c}
A(x) & B_I(x|t) \\
\hline -B_I(x|t)^t & 0\end{array}\right].
\]
Then the Ivanov-Nimmo formula is
\begin{align}\label{e.ivanov}
  P_\lambda(x|t) = \frac{\Pf(A_I(x|t))}{\Pf(A(x))}.
\end{align}

By Schur's identity, the denominator is
\[
  \Pf(A(x)) = \prod_{i<j} \frac{x_i-x_j}{x_i+x_j}.
\]
Writing $\Pf_K(A)$ for the Pfaffian of the submatrix on any subset of rows and columns $K\subseteq [n]$, note that
\[
  \Pf_K(A(x)) = \prod_{k<k' \text{ in }K} \frac{x_k-x_{k'}}{x_k+x_{k'}}.
\]

Now we can state the theorem.  The setting is analogous to that for ordinary Grassmannians.  For each even subset $K\subseteq [n]$, there is a fixed point $p_K\in OG(n)$ and corresponding restriction homomorphism $H_T^*OG(n) \to \C[t]$.  On the other hand, there is a presentation of the cohomology in terms of symmetric functions, i.e., a surjective homomorphism $\C[t][x_0,\ldots,x_{n-1}]^{S_n} \to H_T^*OG(n)$, and composition with the fixed-point restriction to $p_K$ is the evaluation
\[
 \C[t][x_0,\ldots,x_{n-1}]^{S_n} \to \C[t]
\]
given by $x_i \mapsto -t_i$ if $i\in K$, and $x_i\mapsto 0$ if $i\not\in K$.

\begin{theorem}[Type D equivariant Giambelli formula \cite{in}]\label{t.giambelliD}
Under the homomorphism $\C[t][x_0,\ldots,x_{n-1}]^{S_n} \to H_T^*OG(n)$, we have $P_\lambda(x|t) \mapsto \sigma_I$, where $I$ is the subset corresponding to the strict partition $\lambda$.  Equivalently, for each $K=\{k_1<\cdots<k_r\}$, we have
\[
  \sigma_I|_{p_K} = P_\lambda(x|t)|_{x=-t_K},
\]
where the specialization $x=-t_K$ means $x_i \mapsto -t_i$ if $i\in K$, and $x_i \mapsto 0$ if $i\not\in K$.
\end{theorem}

\begin{proof}
We follow the same outline as in type A.  The first step has been done in the previous section, where we worked out the cohomology of quadrics and the corresponding change of basis between $\e_i$ and $f_i$.  The second step has been done above: we identify the spinor $\e_I \in \spin^+[t] = V_{\varpi_n}(-\varpi_n)[t]$ with the Schubert class $\sigma_I \in H_T^*OG(n)$.

It remains to compute the idempotent basis, and the expansion of $\e_I$ in this basis.  With $f_i$ as before, we define spinors $f_I$ by the same formula \eqref{e.spin-basis-e} defining $\e_I$:
\begin{equation}\label{e.spin-basis-f}
  f_I = f_{\bar\imath'_1} \cdots f_{\bar\imath'_{n-r}}\cdot f,
\end{equation}
where $f=f_0\cdots f_{n-1}$.  In fact, $f=\e$, since the change of basis is unitriangular.

Since $\e_\emptyset= \sigma_\emptyset = {\bf1}\in H_T^*OG(n)$, we compute this case first (to normalize the $\barf_I$ basis).  Using notation of the previous section,
\[
  \e_{\bar\imath} = \bar{C}\cdot f_{\bar\imath} + C\cdot f_i.
\]
Since $\bar{C}$ is unitriangular, if we multiply by its inverse and instead consider
\[
  \tilde\e_{\bar\imath} = f_{\bar\imath} + \bar{C}^{-1}C\cdot f_i,
\]
we have
\[
  \e_\emptyset = \e_{\bar0}\cdots \e_{\bar{n-1}}\cdot \e = \tilde\e_{\bar0}\cdots \tilde\e_{\bar{n-1}}\cdot \e
\]
in $\spin^+$.  On the other hand, now we are in the situation of the Theorem of the Appendix, which says
\[
 \e_\emptyset = \sum_{K} \Pf_K(\bar{C}^{-1}C) f_K.
\]

To compute these Pfaffians, it helps to introduce the $n\times n$ diagonal matrix
\[
  S = (t_0-t_{n-1})\cdots(t_{n-2}-t_{n-1}),\, (t_0-t_{n-1})\cdots(t_{n-3}-t_{n-2}),\, \cdots,\, t_0-t_1, 1).
\]
Then $S \bar{C}^{-1}CS = A = (a_{ji})$, where $a_{ji} = \frac{t_j-t_i}{t_j+t_i}$.  So for any $K\subseteq [n]$,
\begin{align*}
 \Pf_K(A) = \Pf_K(S\bar{C}^{-1}CS) = \det_K(S)\cdot\Pf_K(\bar{C}^{-1}C),
\end{align*}
where $\det_K(S)$ means the determinant of the submatrix on rows and columns $K$.  Combining the formulas
\[
  \det_K(S) = \mathop{\prod_{i<k}}_{k\in K}(t_i-t_k) \quad \text{ and } \Pf_K(A) = \mathop{\prod_{i<j}}_{i,j\in K} \frac{t_j-t_i}{t_j+t_i},
\]
we obtain
\[
  \Pf_K(\bar{C}^{-1}C) = \frac{1}{\displaystyle{\mathop{\prod_{i<k}}_{i,k\in K}} (-t_i-t_k) \displaystyle{ \mathop{\prod_{i<k}}_{i\not\in K,k\in K}} (t_i-t_k)}.
\]
Therefore, the idempotent classes $\barf_K$ are determined by writing
\[
 \e_\emptyset = \sum_K \barf_K,
\]
that is, by setting
\begin{align*}
  \barf_K &= \Pf_K(\bar{C}^{-1}C) f_K.
\end{align*}
Equivalently,
\begin{align*}
 f_K &=  \left(\mathop{\prod_{i<k}}_{i,k\in K} (-t_i-t_k)   \mathop{\prod_{i<k}}_{i\not\in K,k\in K} (t_i-t_k)\right) \barf_K.
\end{align*}

Now we compute $\e_I$.  Using notation from our computations at the end of \S\ref{s.quadrics}, the transition between $\e_i$ and $f_i$ has the form
\[
 \left[\begin{array}{c|c}
C & \bar{B} \\
\hline \bar{C} & 0\end{array}\right].
\]
Since the Theorem of the Appendix computes $\e_I$ as $(-1)^{|I|}\e_{i_1}\cdots\e_{i_r}\cdot\e_\emptyset$, we may replace $\e_{\bar\imath}$ by $\tilde\e_{\bar\imath}$, which amounts to using the matrix
\[
 \left[\begin{array}{c|c}
\bar{C}^{-1}C & \bar{B} \\
\hline w_\circ & 0\end{array}\right],
\]
where $w_\circ$ is the $n\times n$ matrix with $1$'s on the antidiagonal and $0$'s elsewhere.  
Now formula \eqref{e.thmA} says
\begin{equation}\label{e.eI}
 \e_I = \sum_{K} \Pf_K (\bar{A}_I) f_K,
\end{equation}
where
\[
 \bar{A}_I =  \left[\begin{array}{c|c}
\bar{C}^{-1}C & \bar{B}_I \\
\hline -\bar{B}_I^t & 0\end{array}\right]
\]
and $\bar{B}_I$ is the submatrix of $\bar{B}$ on columns $I$.

(Note that unitriangularity of the matrix relating $\e_i$ and $f_i$ implies that
\begin{align*}
  \e_I &= f_I + \sum_{K> I} \Pf_K(\bar{A}_I) f_K \\
  &= \left(\mathop{\prod_{j<i}}_{i,j\in I} (-t_j-t_i)   \mathop{\prod_{j<i}}_{j\not\in I,i\in I} (t_j-t_i)\right) \barf_I + \cdots ,
\end{align*}
which agrees with the formula \eqref{e.OGrestrict} for $\sigma_I|_{p_I}$.)

To conclude the proof, we must relate the coefficients $\Pf_K(\bar{A}_I)$ to the evaluations $P_\lambda(x|t)|_{x=-t_K}$.  Observe first that
\begin{align*}
 \Pf(A(x))|_{x=-t_K} &= \prod_{k'<k\text{ in }K}\frac{-t_{k'}+t_k}{-t_{k'}-t_k} \\
                     &=  \Pf_K(A) \\
                     &= \det_K(S)\Pf_K(\bar{C}^{-1}C).
\end{align*}
Furthermore, scaling the first $n$ rows and columns of $\bar{A}_I$ by $S$, one computes
\begin{align*}
  \left[\begin{array}{c|c}
S\bar{C}^{-1}CS & S\bar{B}_I \\
\hline -\bar{B}_I^tS & 0\end{array}\right]
&=   \left.\left[\begin{array}{c|c}
A(x) & B_I(x|t) \\
\hline -B_I(x|t)^t & 0\end{array}\right]\right|_{x_0=-t_0,\ldots,x_{n-1}=-t_{n-1}},
\end{align*}
and also that
\begin{align*}
 & \Pf_K \left( \left.\left[\begin{array}{c|c}
A(x) & B_I(x|t) \\
\hline -B_I(x|t)^t & 0\end{array}\right]\right|_{x_j=-t_j, \text{ all }j}\right) \\ &\qquad\qquad= \left.\left(\Pf \left[\begin{array}{c|c}
A(x) & B_I(x|t) \\
\hline -B_I(x|t)^t & 0\end{array}\right]\right)\right|_{x=-t_K}.
\end{align*}
Taking Pfaffians on rows and columns $K$, it follows that
\[
  \Pf_K(\bar{A}_I)\cdot \det_K(S) = \left.\left(\Pf \left[\begin{array}{c|c}
A(x) & B_I(x|t) \\
\hline -B_I(x|t)^t & 0\end{array}\right]\right)\right|_{x=-t_K}.
\]
So the Ivanov-Nimmo formula gives
\begin{align*}
 P_\lambda(x|t)|_{x=-t_K} &= \frac{\Pf_K(\bar{A}_I)}{\Pf_K(\bar{C}^{-1}C)},
\end{align*}
which is precisely the coefficient of $\barf_K = \Pf_K(\bar{C}^{-1}C) f_K$ in the expansion of $\e_I$ in \eqref{e.eI}.
\end{proof}

\section{Rim hook rules for quantum cohomology}\label{s.rim-hook}

Consider a homogeneous space $X=G/P$ such that $H^2(X) \isom \C$.  (This includes all minuscule $G/P$.)  The (small) quantum cohomology ring of $X$ is a $\C[q]$-algebra $QH^*(X) = \C[q]\otimes H^*(X)$, where $q$ is a formal parameter whose degree depends on $X$, equipped with a product which deforms the usual cup product on $H^*(X)$; the structure constants in $QH^*(X)$ are {\it $3$-point Gromov-Witten invariants}, counting certain rational curves in $X$.  Using the action of a maximal torus $T\subseteq G$ on $X$, there are also equivariant versions $QH_T^*(X)$.  We will write $T = (\C^*)^n$, and
\[
 \htpt = \htpt_n = H_{T}^*(\pt) \isom \Z[t_1,\ldots,t_n],
\]
so $QH_T^*(X)$ is an algebra over $\htpt[q]$.

For any $N$, consider the embedding of Grassmannians
\[
 \iota\colon \Gr(k,N) \hookrightarrow \Gr(k,N+1)
\]
corresponding to the standard embedding of $\C^N$ in $\C^{N+1}$ (as the span of the first $N$ standard basis vectors).  Let us write $T_N=(\C^*)^N$, and $\htpt_N=H_{T_N}^*(\pt)$.  There is a similar embedding of $T_N$ in $T_{N+1}$ (corresponding to the map on character groups $\Z^{N+1} \to \Z^N$ sending $t_{N+1}$ to $0$ and all other $t_i$ to $t_i$).  The embedding $\iota$ is equivariant with respect to the inclusion of tori and their natural actions on the Grassmannians, so we have an inverse system of graded homomorphisms $\iota^*\colon H_{T_{N+1}}^*\Gr(k,{N+1}) \to H_{T_N}^*\Gr(k,N)$.  Let $H_{T_\infty}^*\Gr(k,\infty)$ be the graded inverse limit.  This ring can be identified with the ring of symmetric polynomials in variables $x_1,\ldots,x_k$, with coefficients in $\htpt_\infty$.

Writing $\Omega_\lambda^{(N)}$ for the Schubert variety in $\Gr(k,N)$ and $\sigma_\lambda^{(N)}$ for its equivariant cohomology class, one checks that $\iota^{-1}\Omega_\lambda^{(N+1)} = \Omega_\lambda^{(N)}$, and it follows that $\iota^*\sigma^{(N+1)}_\lambda = \sigma^{(N)}_\lambda$.  Therefore we have well-defined classes $\sigma_\lambda \in H_{T_\infty}^*\Gr(k,\infty)$.

Fixing $n$ and $T=T_n$, for $N>n$, consider a (different) inclusion of tori $T \hookrightarrow T_N$ given by
\[
  (z_1,z_2,\ldots,z_n) \mapsto (z_1,z_2,\ldots,z_n,z_1,z_2,\ldots);
\]
equivalently, take the map on character groups $\Z^N \to \Z^n$ given by $t_{i} \mapsto t_{i \pmod n}$, where the representatives mod $n$ are taken to be $1,\ldots,n$.  This also defines a ring homomorphism $\htpt_\infty \to \htpt_n$, inducing an algebra homomorphism $H_{T_\infty}^*\Gr(k,\infty) \to H_{T}^*\Gr(k,\infty)$.  We will see how to extend this to a homomorphism $H_{T}^*\Gr(k,\infty) \to QH_{T}^*\Gr(k,n)$.

The {\it rim hook rule}\footnote{The non-equivariant rim hook rule was discovered by Bertram, Ciocan-Fontanine, and Fulton during the 1996--7 program on quantum cohomology at Institut Mittag-Leffler \cite{bcff}.  Bertiger, Mili\'cevi\'c, and Taipale gave an equivariant generalization \cite{bmt}.} for the equivariant quantum cohomology of $\Gr(k,n)$ is a recipe for computing quantum products in terms of ordinary products on $\Gr(k,N)$, for $N>n$.  Given any partition $\lambda$, an \define{$n$-rim hook} (also called an \define{$n$-border strip}) is a connected collection of $n$ boxes on the southeast border of its diagram, with exactly one box in each diagonal, such that the complement $\mu$ is also the diagram of a partition.  If the boxes appear in rows $i$ through $j$ of the diagram, the \define{height} of the rim hook is $j-i$.  See Figure~\ref{f.rimhook1} for an example, and \cite[\S I.1]{macdonald} for more details.

Any partition $\lambda$ has a well-defined \define{$n$-core} $\mu$, which is the partition obtained by removing all possible $n$-rim hooks from $\lambda$, in any order.  The parity of the sum of the heights of these rim hooks is also independent of choices, so we can define $\epsilon(\lambda/\mu) = \sum_{\delta} \mathrm{height}(\delta) \pmod 2$, taking the sum over a pieces of a decomposition of the skew shape $\lambda/\mu$ into $n$-rim hooks $\delta$.

Under the bijection between partitions $\lambda$ with $k$ parts and $k$-element subsets $I$ of $\{1,2,\ldots\}$, the operation of removing an $n$-rim hook corresponds to replacing some element $i\in I$ with $i-n$ (and an $n$-rim hook exists only if there is an $i>n$ such that $i-n$ is not in $I$); the height of such a rim hook is the length of the permutation needed to sort the resulting set into increasing order.  Taking the $n$-core of $\lambda$ can be described as follows.  Write the elements of $I$ as $s_i n + r_i$, for $1\leq i\leq k$ and $1\leq r_i\leq n$.  Next consider the multiset of residues $\{r_i\}$.  For each $r$ that appears in this multiset, replace the $j$th occurrence of $r$ by $r+(j-1)n$, obtaining a set $\bar{I}$ of $k$ distinct integers with the same multiset of residues modulo $n$.  The $n$-core of $\lambda$ is the partition $\mu$ corresponding to $\bar{I}$, and the sign $(-1)^{\epsilon(\lambda/\mu)}$ is the sign of the permutation needed to sort $\bar{I}$ (see \cite[\S I.1, Ex.~8]{macdonald}).

For example, consider $\lambda = (7,6,3)$.  For $k=3$, this corresponds to the set $I=\{4,8,10\}$.  There are two ways to remove a $7$-rim hook: replace $10$ with $3$ (obtaining $I' = \{4,8,3\}$) or replace $8$ with $1$ (obtaining $I'=\{4,1,10\}$).  In the first case, the permutation which sorts $I'$ has length $2$, and in the second it has length $1$.  To find the $7$-core, write residues modulo $7$ to obtain $\{4,1,3\}$; this is sorted to $\bar{I}=\{1,3,4\}$ by a permutation of length $2$, so $\mu = (1,1)$ and $\epsilon(\lambda/\mu)$ is even.

\begin{figure}
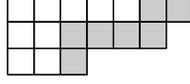



\pspicture(60,40)(-100,-10)

\psset{unit=1.00pt}

\whitebox(-50,10)
\whitebox(-40,10)
\whitebox(-30,10)
\whitebox(-20,10)
\whitebox(-10,10)
\graybox(0,10)
\graybox(10,10)

\whitebox(-50,0)
\whitebox(-40,0)
\graybox(-30,0)
\graybox(-20,0)
\graybox(-10,0)
\graybox(0,0)

\whitebox(-50,-10)
\whitebox(-40,-10)
\graybox(-30,-10)

\endpspicture


\caption{A $7$-rim hook of height $2$.\label{f.rimhook1}}

\end{figure}

Given a class $\sigma_\lambda\in H_{T_\infty}^*\Gr(k,\infty)$, for $\lambda = (\lambda_1\geq \cdots \geq \lambda_k \geq 0)$, let $\mu$ be the $n$-core of $\lambda$ and define a linear map
\begin{align*}
  \phi\colon H_{T_\infty}^*\Gr(k,\infty) &\to QH_{T}^*\Gr(k,n) \\
       \sigma_\lambda & \mapsto \begin{cases} (-1)^{(k-1)s+\epsilon(\lambda/\mu)}\,q^s\,\sigma_\mu & \text{ if } \mu \subseteq \rho_{k,n-k}, \\ 0 & \text{ otherwise,} \end{cases}
\end{align*}
where $s = (|\lambda|-|\mu|)/n$ is the number of rim hooks removed from $\lambda$ to obtain $\mu$.
This map factors into the specialization homomorphism $H_{T_\infty}^*\Gr(k,\infty) \to H_{T}^*\Gr(k,\infty)$ described above, followed by a linear map $\psi\colon H_{T}^*\Gr(k,\infty) \to QH_{T}^*\Gr(k,n)$ which is given by the same formula as $\phi$.  The first is clearly an algebra homomorphism (with respect to the ring map $\htpt_{\infty} \to \htpt$ which cyclically specializes the $t$ variables), while the second is {\it a priori} a homomorphism of $\htpt$-modules.

\begin{theorem}[Equivariant rim hook rule \cite{bcff,bmt}]\label{t.rimhook}
The map $\phi$ respects multiplication: it is a surjective homomorphism of algebras, compatible with the cyclic specialization $\htpt_\infty \to \htpt$.
\end{theorem}

It suffices to show that the second factor $\psi$ is a surjective ring homomorphism.  We will give a simple proof, inspired by the Satake correspondence.  (A similar construction was described by Gatto \cite{gatto}.  The phrasing in terms of reduction modulo $n$ also appears in work by Buch \cite[Corollary~1]{buch} and Sottile \cite{sottile}.)  The basic idea is to show that the kernel of $\psi$ is an ideal, so that $\psi$ induces a ring structure on its image; then apply Mihalcea's characterization of the quantum product \cite{mihalcea1} to conclude that the product induced by $\psi$ is the quantum product.

Before proving the theorem, we first consider the easier case of projective space, so $k=1$.  Writing a class in $H_{T}^*\P^{\infty}$ as $\sigma_{sn+i}$, for $0\leq i\leq n-1$, the rim hook rule says $\psi(\sigma_{sn+i}) = q^s\sigma_{i}$.  (In this case, $\psi$ is an isomorphism of $\htpt$-modules.)  Both $H_{T_n}^*\P^{\infty}$ and $QH_{T}^*\P^{n-1}$ are free $\htpt$-algebras, generated by the divisor class $\sigma_{\bx} = \sigma_1$, so to see that $\psi$ is an isomorphism of algebras it suffices to check that it respects multiplication by $\sigma_{\bx}$.  This is a simple application of a special case of the equivariant (quantum) Pieri rule \cite{mihalcea1}: in $H_{T}^*\P^{\infty}$ one has
\[
  \sigma_{\bx}\cdot \sigma_{i} = \sigma_{i+1} + (t_1-t_{(i+1)\pmod n})\,\sigma_{i},
\]
and in $QH_{T_n}^*\P^{n-1}$ one has
\[
  \sigma_{\bx}\cdot \sigma_{i} = \begin{cases} \sigma_{i+1} + (t_1-t_{i+1})\, \sigma_{i} & \text{ for } i<n-1, \\ q\,\sigma_{0} + (t_1-t_{n})\,\sigma_{n-1} & \text{ for } i=n-1. \end{cases}
\]

This computation can be rephrased to make it analogous to the one we did in the introduction.  There are homomorphisms with compatible identifications
\begin{equation}\label{e.qtwedge1}
\begin{tikzcd}
 H_{T_\infty}^*\P^\infty  \ar[r] \ar[d,equal] & H_{T}^*\P^{\infty}  \ar[r,"\psi"] \ar[d,equal] & QH_{T}^*\P^{n-1} \ar[d,equal] \\
 (\htpt_\infty)^\infty   \ar[r] & (\htpt)^{\infty} \ar[r,"\psi"] & (\htpt[q])^n,
\end{tikzcd}
\end{equation}
where the first horizontal map is given by the cyclic specialization $t_i \mapsto t_{i\pmod n}$ as above.  Identifying the standard basis $\e_i$ with $\sigma_{i-1}$ as before, the rim hook rule says $\psi$ is given by $\e_{sn+i} \mapsto q^s\e_i$, for $1\leq i\leq n$.  (Recalling that $\e_i$ is a weight vector with weight $t_i$, one sees that the cyclic specialization is necessary to make the second map compatible with the $S_n$-module structure.)

Multiplication by $\sigma_{\bx}$ on $H_{T_\infty}^*\P^\infty$ is given by the infinite matrix $\xi_t$ with $-t_1,-t_2,\ldots$ on the diagonal, $1$'s on the subdiagonal, and $0$'s elsewhere.  (One should twist $(\htpt_\infty)^\infty$ by $\ee^{-t_1}$ as in \S\ref{s.eq-satake} to get the correct action.)  Specializing the diagonal variables modulo $n$ and applying the homomorphism $\psi$, the action of $\xi_t$ on $(\htpt_\infty)^\infty$ is transformed into the action of the $n\times n$ matrix
\begin{equation}\label{e.xiqt}
 \xi_{q,t} = \left[\begin{array}{cccc}
-t_1 &  &  & q \\
1 & \ddots &  &  \\
0 & \ddots & \ddots &  \\
0 & 0 & 1 & -t_n\end{array}\right]
\end{equation}
on $(\htpt[q])^n$, which agrees with multiplication by $\sigma_{\bx}$ in $QH_{T}^*\P^{n-1}$.

\begin{lemma}
The $\htpt$-module homomorphism $\psi\colon H_{T}^*\Gr(k,\infty) \to QH_{T}^*\Gr(k,n)$ is surjective, and its kernel is an ideal.
\end{lemma}

\begin{proof}
It is straightforward to see that $\psi$ is surjective, and we leave this to the reader.  It is also easy to see that the kernel is generated (as an $\htpt$-module) by two types of elements:
\begin{enumerate}
\item classes $\sigma_\lambda$ such that the $n$-core of $\lambda$ does not fit in $\rho_{k,n-k}$; and

\smallskip

\item differences $(-1)^{\epsilon(\lambda/\mu)}\sigma_\lambda - (-1)^{\epsilon(\lambda'/\mu)}\sigma_{\lambda'}$, for two partitions $\lambda,\lambda'$ of the same size and with the same $n$-core $\mu$.
\end{enumerate}
We will show that the products of such elements with $\htpt$-algebra generators for $H_T^*\Gr(k,\C^\infty)$ are also in the kernel of $\psi$.

Making identifications
\begin{equation}\label{e.qtwedge2}
\begin{tikzcd}
H_{T}^*\Gr(k,\infty) \ar[r,"\psi"] \ar[d,equal] & QH_{T}^*\Gr(k,n) \ar[d,equal]\\
 \exterior^k_{\htpt}\htpt^{\infty} \ar[r,"\psi"] & \exterior^k_{\htpt[q]}(\htpt[q])^n,
\end{tikzcd}
\end{equation}
up to sign the homomorphism $\psi$ is induced by the corresponding map described above for projective space: if one sends $\e_{sn+i}$ to $((-1)^{k-1}q)^s\,\e_i$, then $\psi$ is the induced map on the $k$th exterior power.  Furthermore, from this point of view, the operators $\xi_t,\xi_t^2,\ldots,\xi_t^k$ ---or rather, the cohomology classes they correspond to---form a set of $\htpt$-algebra generators for $H_T^*\Gr(k,\infty)$.  (To see this, recall that $H_T^*\Gr(k,\infty)$ is isomorphic to the ring of symmetric polynomials in $k$ variables with coefficients in $\htpt$, and that the leading term of $\xi_t^j$ corresponds to multiplication by the power sum function $p_j$; these generate the ring of symmetric polynomials, for $1\leq j\leq k$.)

Now let us consider the generators of the kernel of $\psi$, using the correspondence between partitions and $k$-element sets to write classes in $H_T^*\Gr(k,\infty)$ as $\sigma_\lambda = \e_I$.  For $I=\{i_1<\cdots<i_k\}$, extract residues by writing $i_a = s_a n + r_a$.  The $n$-core of $\lambda$ fits inside $\rho_{k,n-k}$ if and only if the residues $r_a$ are all distinct, and two partitions $\lambda,\lambda'$ share the same $n$-core if and only if the corresponding sets $I,I'$ share the same multiset of residues.

Suppose $\e_I$ is the first type of generator for the kernel, so $I$ has at least two elements $i,i'$ with the same residue modulo $n$.  Reordering as necessary, we can write
\[
  \e_I = \pm (\e_{i}\wedge \e_{i'} \wedge \cdots ).
\]
Using the formula from Example~\ref{ex.powers},
\begin{align*}
  \xi_t^j\cdot (\e_{i}\wedge \e_{i'} \wedge \cdots ) &= \left(\sum_{a=0}^j (-1)^a\,h_a(t_{i},\ldots,t_{i+j-a})\,\e_{i+j-a}\right) \wedge \e_{i'} \wedge \cdots \\
    & \qquad + \e_i \wedge \left(\sum_{a=0}^j (-1)^a\,h_a(t_{i'},\ldots,t_{i'+j-a})\,\e_{i'+j-a}\right) \wedge \cdots \\
    & \qquad + (\text{ terms involving } \e_i \wedge \e_{i'} \wedge \cdots ) ;
\end{align*}
Due to the cyclic specialization, the subscripts on the $t$ variables should be read modulo $n$, and since $i\equiv i' \pmod n$, the coefficients $h_a(t_{i},\ldots,t_{i+j-a})$ and $h_a(t_{i'},\ldots,t_{i'+j-a})$ are equal.  After applying $\psi$, the first two terms cancel, and all the others go to zero.  It follows that $\xi_t^j\cdot \e_I$ also lies in the kernel.

The second case is similar.  Recall that the sign $(-1)^\epsilon(\lambda/\mu)$ is equal to the sign of the permutation needed to sort the residues of $I$ into increasing order.  Suppose $I$ and $I'$ have the same (distinct) residues modulo $n$, and are such that the correpsonding partitions $\lambda$ and $\lambda'$ have the same size.  Then
\[
  (-1)^{\epsilon(\lambda/\mu)}\,\e_I - (-1)^{\epsilon(\lambda'/\mu)}\,\e_{I'} = (\e_{s_1n+r_1}\wedge \cdots \wedge \e_{s_kn+r_k}) - (\e_{s'_1n+r_1}\wedge \cdots \wedge \e_{s'_kn+r_k}),
\]
for $r_1<\cdots<r_k$.  A calculation analogous to the previous one shows that the product of this difference with $\xi_t^j$ remains in the kernel of $\psi$.
\end{proof}

\begin{proof}[Proof of Theorem~\ref{t.rimhook}]
As remarked above, it suffices to show that $\psi$ is a ring homomorphism.  By the lemma, $\ker(\psi)$ is an ideal of $H_T^*\Gr(k,\infty)$.  Let $A=H_T^*\Gr(k,\infty)/\ker(\psi)$, and write $q\in A$ for the image of $(-1)^{k-1}\sigma_{n}$.  This makes $A$ an $\htpt[q]$-algebra, with an $\htpt[q]$-module basis indexed by partitions fitting inside $\rho_{k,n-k}$.  The homomorphism $\psi$ maps this basis onto the Schubert basis for $QH_T^*\Gr(k,n)$.  We must check that this is an isomorphism of algebras.

By Mihalcea's characterization of $QH_T^*\Gr(k,n)$ \cite[Corollary~7.1]{mihalcea1}, it suffices to check that the isomorphism respects multiplication by the divisor class $\sigma_1$.  On one hand, multiplication by $\sigma_1$ on $A$ is given by the action of a matrix $\xi_{q,t}$ on $\exterior_{\htpt[q]}^k(\htpt[q])^n\otimes \ee^{-\varpi_k}$.  (The matrix $\xi_{q,t}$ is the one of \eqref{e.xiqt}, with $q$ replaced by $(-1)^{k-1}q$.)  That is,
\begin{align*}
  \xi_{q,t}\cdot\e_I &= \e_{i_1+1}\wedge\e_{i_2}\wedge\cdots \wedge \e_{i_k} + \cdots + \e_{i_1}\wedge\e_{i_2}\wedge\cdots \wedge \e_{i_k+1} \\
   & \qquad + (t_1+\cdots+t_k-t_{i_1}-\cdots-t_{i_k})\,\e_I \\
   & \qquad (\,+\, q \e_1 \wedge \e_{i_1} \wedge \cdots \wedge \e_{i_{k-1}}),
\end{align*}
with the last term occurring only if $i_k=n$ and $i_1>1$.  Writing this in terms of partitions, one recovers exactly the equivariant quantum Pieri formula \cite[Theorem~1]{mihalcea1} for multiplication by $\sigma_1$ in $QH_T^*\Gr(k,n)$, as desired.\footnote{Our indexing of $t$ variables is reversed when compared with that of \cite{mihalcea1}, since we use opposite Schubert classes.}
\end{proof}

\section{Closing remarks and other directions}\label{s.closing}

\subsection{Relation to quantum integrability}
In \cite{gk,gks}, Gorbounov,  Korff, and Stroppel indicate an alternative approach to the (equivariant) rim-hook rule, based on quantum integrability of the six-vertex model. This model can be seen as a degeneration of the Yangian Hopf algebra (a certain quantum group) which governs the cohomology of the cotangent space of the Grassmannian. This fits in a major program initiated by Maulik and Okounkov to relate representation theory of quantum groups to enumerative geometry  \cite{mo}, \cite{o}, \cite{rsvz}.

The authors of \cite{gk,gks} consider a vector space $V=\C v_0\oplus \C v_1$.  The standard basis $\left\{v_{\omega_1}\otimes v_{\omega_1}\otimes \ldots \otimes v_{\omega_n}\right\}$ of $V^{\otimes n}$, where $\omega_i\in \left\{0,1\right\}$, can be identified with set of $01$-words $\left\{\omega=\omega_1\omega_2\ldots\omega_n\right\}$.  The set of words $\left\{\omega : \sum_{i=1}^{n}\omega_i=k\right\}$ is in bijection with the set of multi-indexes $I = \{i_1 < \cdots < i_k\} \subseteq [n]$.  The bijection is given by sending $I$ to the $01$-word with $1$'s in positions $i_1,\ldots,i_k$ and $0$'s elsewhere.  Through this identification, let $V_k\subseteq V^{\otimes n}$ be the vector subspace generated by vectors indexed by this set.  In \cite{gks}, two bases for $V_k\otimes \C[t_1,\ldots,t_n]$ are constructed, which are called respectively standard or spin basis\footnote{This ``spin basis'' is not directly connected to spin representations and the orthogonal Grassmannian, as far as we are aware.} and the Bethe vector basis. Then a connection is made with equivariant cohomology of the Grassmannian, by identifying those bases with the Schubert basis and the torus fixed point basis, respectively.  The authors construct certain (generating series of) operators denoted $A+qD$ (\cite[\S5]{gks}) acting on the spin basis which match the action of quantum multiplication by the divisor class on the Schubert basis.

This is analogous to the setup from the previous section, where the principal nilpotent $\xi_t$ is corrected by the nilpotent matrix
\[
q\left[
  \begin{array}{cccc}
0 &  &  & 1 \\
0 & \ddots &  &  \\
0 & \ddots & \ddots &  \\
0 & 0 & 0 & 0\end{array}
\right]
\]
in order to match the quantum equivariant Chevalley rule after a twist. 

We should also mention another closely related result by Lam and Templier, \cite[Proposition~4.13]{lt}  where essentially the same presentation of the quantum Chevalley operator is given.

\subsection{Mirror symmetry and the Gamma Conjectures}

Mirror symmetry for Fano varieties roughly claims a correspondence which associate to each Fano variety $X$ a smooth quasi-projective variety $Y$ endowed with  a
flat map $w:Y\longrightarrow \mathbb{A}^1$ with quasi-projective fibres, called superpotential. Under this correspondence critical points of $w$ give rise to vanishing cycles which through homological mirror symmetry correspond to objects of the derived category $D^b(X)$ of coherent sheaves on $X$.

The Gamma conjectures were formulated by Galkin, Golyshev and Iritani in \cite{ggi}, refining previous conjectures by Dubrovin \cite{dub}.  They claim that simple eigenvalues of a certain variation of the quantum Chevalley operator denoted $\star_0$ are related to characteristic classes $\hat{\Gamma}_X Ch(E)$, where $E$ is an exceptional object of $D^b(X)$, which in turn are constructed through flat sections of Dubrovin's quantum connection
\begin{equation}\label{qc}
\nabla_{z\partial_z}=z\frac{\partial}{\partial z}-\frac{1}{z}(c_1(X)\star_0)+\mu,
\end{equation}
a meromorphic flat connection on the trivial bundle $H^*(X)\times \mathbb{P}^1\rightarrow \mathbb{P}^1$. In (\ref{qc}) $\mu$ is the grading operator on $H^*(X)$ and $z$ is a local coordinate on $\mathbb{P}^1$.

The Gamma conjectures were proved for Grassmannians in \cite{ggi}, (see also \cite{cdg}), by making fundamental use of the quantum correction to the Satake correspondence.

\subsection{Other representations}
We have focused on minuscule representations in types A and D, with their corresponding (minuscule) homogeneous spaces $G/P$: Grassmannians, quadrics, and orthogonal Grassmannians.  It is natural to ask what can be said for other spaces.

In other types ($E_6$ and $E_7$), there is only one minuscule space, up to isomorphism, so the approach we used in Sections~\ref{s.giambelli} and \ref{s.spinors} does not seem productive.  Isomorphisms between quantum cohomology of these spaces and corresponding representations were worked out by Golyshev and Manivel \cite{gm}.  On the other hand, it would be interesting to see an analogue of the rim-hook rule for orthogonal Grassmannians, presumably related to an infinite-dimensional spin representation.

In a different direction, one can ask for the transition matrices between the ``$\e$'' and ``$\barf$'' bases of non-minuscule representations.  This should correspond to localization formulas for MV-cycles.  We have seen Schur $S$- and $P$-functions appear naturally in the exterior and spin representations; what other symmetric functions arise this way?

\appendix
\section*{Appendix: Pfaffians and spinors}
\setcounter{equation}{0}
\renewcommand{\theequation}{A.\arabic{equation}}

In this appendix, we prove a change-of-basis formula in the spin representation, where Pfaffians play the role analogous to determinants in the exterior algebra.  This refines similar formulas of Chevalley and Manivel \cite{chevalley,manivel}.

First, for any complex vector space $V$ with symmetric bilinear form $\langle\;,\;\rangle$, the associated \define{Clifford algebra} is the quotient of the tensor algebra by two-sided ideal forcing the relation $v\cdot w + w\cdot v = \langle v,w\rangle 1$ for all vectors $v,w\in V$:
\[
  Cl(V) := T^\bullet(V) / ( v\otimes w + w\otimes v - \langle v,w\rangle 1).
\]
(Note that this definition differs slightly from the standard one, where $\langle v,w \rangle 1$ would be replaced by $2\langle v,w \rangle 1$ --- but it is equivalent, up to rescaling the form $\langle \;,\; \rangle$, since we are not in characteristic $2$.)  If the bilinear form is zero, this is just the exterior algebra: $Cl(V) \isom \exterior^\bullet V$.  In general, $\dim Cl(V) = 2^{\dim V}$, with a basis consisting of products $v_I=v_{i_1}\cdots v_{i_k}$ of distinct basis elements of $V$. (One can prove this by degenerating to the exterior algebra.)

We note the following general formulas in $Cl(V)$, which are immediate from the defining relations.  First, if $x$ and $y$ are orthogonal vectors in $V$ (i.e., $\langle x,y \rangle = 0$), then
\begin{align}
  x\cdot y &= -y\cdot x
\end{align}
in $Cl(V)$.  Next, suppose $x = x_1\cdots x_r \in Cl(V)$ is a product of vectors in $V$ such that $X=\Span\{x_1,\ldots,x_r\} \subseteq V$ is isotropic.  Then for any vector $y\in X$,
\begin{align}
 y\cdot x = 0
\end{align}
in $Cl(V)$.

From now on, we assume $\dim V=2n$ and its bilinear form is nondegenerate.
Let $y_{\bar{n-1}},\ldots,y_{\bar{0}},y_0,\ldots,y_{{n-1}}$ be an ``orthogonal basis'', meaning $\langle y_{\bar\imath}, y_j \rangle = \delta_{ij}$.  (We interpret the bar as a negative sign, so $\bar{\bar\imath}=i$.)  Let $y=y_0\cdots y_{n-1}$.  The \define{spin representation} is the left ideal
\[
  \spin = Cl(V)\cdot y.
\]
A \define{(pure) spinor} is an element of the form $z\cdot y \in \spin$, where $z = z_1\cdots z_n$ is a product of vectors $z_1,\ldots,z_n \in V$ which span a maximal isotropic subspace of $V$.

Let $[m] = \{0,\ldots,m-1\}$ for any integer $m\geq 1$.  For a subset $I=\{i_1<\cdots<i_r\} \subseteq [n]$, let $I' = \{i'_1<\cdots<i'_{n-r}\} = [n]\setminus I$.  Given such an $I$, the corresponding standard spinor is
\[
  y_I := y_{\bar\imath'_1} \cdots y_{\bar\imath'_{n-r}}\cdot y.
\]
These form a basis for $\spin$, as $I$ ranges over all subsets of $[n]$.  Note that $y_\emptyset = y_{\bar{0}}\cdots y_{\bar{n-1}}\cdot y$, and
\begin{align}\label{e.alt-basis}
  y_{i_1}\cdots y_{i_r}\cdot y_\emptyset = (-1)^{i_1+\cdots+i_r} y_I,
\end{align}
which is sometimes useful.

The spin representation becomes an $\lieso_{2n}$-module via the embedding $\lieso_{2n} \isom \exterior^2 V \hookrightarrow Cl(V)$ by
\[
  v\wedge w \mapsto \frac{1}{2}(v\cdot w - w\cdot v) = v\cdot w - \frac{1}{2}\langle v,w \rangle 1.
\]
So $\lieso_{2n}$ preserves the parity of basis vectors of $Cl(V)$, and the spin representation breaks into two irreducible \define{half-spin representations}
\[
  \spin = \spin^+ \oplus \spin^-.
\]
Our convention will be that $\spin^+$ is spanned by standard spinors $y_I$, with $I$ an even subset of $[n]$, i.e., it has even cardinality.

Now let $x_{\bar{n-1}},\ldots,x_{\bar{0}},x_0,\ldots,x_{{n-1}}$ be another orthogonal basis, and assume it is related to the $y_i$'s by a unitriangular matrix:
\begin{align*}
 x_{\bar\imath} &= y_{\bar\imath} + \sum_{j<i} \bar{c}_{ji} y_{\bar\jmath} + \sum_{j=0}^{n-1} c_{ji} y_j \\
 \intertext{and}
 x_i &= y_i + \sum_{j>i} b_{ji}y_j.
\end{align*}
Thus $x=x_0\cdots x_{n-1}$ is equal to $y=y_0\cdots y_{n-1}$.

We define pure spinors $x_I$ analogously, by
\[
  x_I := x_{\bar\imath'_1} \cdots x_{\bar\imath'_{n-r}}\cdot y.
\]
Our goal is to compute the expansion of $x_I$ in the basis $y_K$.

In fact, we will be free to multiply by the inverse of the matrix $\bar{C}=(\bar{c}_{ji})$ and assume that $x_{\bar\imath}$ is written
\begin{align*}
  x_{\bar\imath} &= y_{\bar\imath}  + \sum_{j=0}^{n-1} a_{ji} y_j.
\end{align*}
In this case, the fact that the $x_{\bar\imath}$ span an isotropic space is equivalent to the fact that the matrix $A=(a_{ji})$ is skew-symmetric.  Up to appropriately indexing rows and columns, the $x$ basis is related to the $y$ basis by a matrix of the form
\[
\left(\begin{array}{c|c}
A & B \\
\hline w_\circ & 0\end{array}\right),
\]
where $w_\circ$ is the matrix with $1$'s on the antidiagonal and zeroes elsewhere.  The top $n$ rows determine a skew-symmetric matrix (by replacing $w_\circ$ with $-B^t$).  For a subset $I\subseteq [n]$, let $B(I)$ be the submatrix formed by taking columns $I$ of $B$, and let $A(I)$ be the skew-symmetric matrix
\[
A(I) = \left(\begin{array}{c|c}
A & B(I) \\
\hline -B(I)^t & 0\end{array}\right).
\]
Only the top $n$ rows are needed to perform operations with such a matrix.  For instance, if $n=3$ and $I=\{1\}$, the top rows are
\[
 \left(\begin{array}{ccc|ccc}
 0 & a_{21} & a_{20} & b_{20} & b_{21} & 1 \\
 & 0 & a_{10} & b_{10} & 1 & 0 \\
 &  & 0 & 1 & 0 & 0\end{array}\right)
\]
and
\[
A(I) = \left(\begin{array}{ccc|c}0 & a_{21} & a_{20} & b_{21} \\
 -a_{21} & 0 & a_{10} & 1 \\
  -a_{20} & -a_{10} & 0 & 0 \\ \hline
  -b_{21} & -1 & 0 & 0\end{array}\right)
\]

For even $r$, the \define{Pfaffian} of any $r\times r$ skew-symmetric matrix $A$ may be computed recursively using the Laplace-type expansion formula
\[
  \Pf(A) = \sum_{j=1}^{r-1} (-1)^{j-1} a_{jr}\Pf_{\hat{j,r}}(A),
\]
where $\Pf_{\hat{j,r}}(A)$ is the Pfaffian of the submatrix of $A$ obtained by removing the $j$th and $r$th rows and columns.  Let $\Pf_K(A)$ denote the Pfaffian of the submatrix on rows and columns $K$, and we always use the convention $\Pf_\emptyset(A)=1$.

\begin{thm*}
For a subset $I\subseteq [n]$, we have
\begin{equation}\label{e.thmA}
  x_I = \mathop{\sum_{K\subseteq [n]}}_{K \text{ even}} \Pf_K(A(I))\, y_K,
\end{equation}
where the sum is over subsets $K$ of even cardinality.
\end{thm*}

See also \cite[Corollary 2.4]{okada} for a related Pfaffian identity.

\begin{proof}
First we establish the formula for $x_\emptyset$.  In fact, for any $m\leq n$, we have
\[
  x_{\bar{0}}\cdots x_{\bar{m-1}}\cdot y = \mathop{\sum_{K\subseteq [m]}}_{K \text{ even}} \Pf_K(A)\, y_{\bar{k}'_1}\cdots y_{\bar{k}'_s}\cdot y,
\]
where $K' = \{k'_1<\cdots<k'_s\} = [m]\setminus K$.  This is proved by induction on $m$, the base case $m=0$ being the tautology $y=y$.  For the induction step, we compute using the Laplace expansion formula:
\begin{align*}
  x_{\bar{0}} \cdots x_{\bar{m}}\cdot y &= (-1)^m x_{\bar{m}} \cdot x_{\bar{0}} \cdots x_{\bar{m-1}}\cdot y \\
  &= (-1)^m \left( y_{\bar{m}} + \sum_{j=0}^{m-1} a_{jm} y_j \right) \sum_{K \subseteq [m]} \Pf_K(A) y_{\bar{k}'_1} \cdots y_{\bar{k}'_r} \cdot y \\
  &= \sum_{K \subseteq [m]} \Pf_K(A) y_{\bar{k}'_1} \cdots y_{\bar{k}'_r}\cdot y_{\bar{m}} \cdot y \\
  & \quad + \sum_{j=0}^{m-1}\sum_{K \subseteq [m]} (-1)^m a_{jm}\Pf_K(A) y_{j}\cdot y_{\bar{k}'_1} \cdots y_{\bar{k}'_r} \cdot y.
\end{align*}
Note
\[
 y_j\cdot y_{\bar{k}'_1} \cdots y_{\bar{k}'_r} \cdot y = \begin{cases} 0 &\text{if }j\in K; \\ (-1)^{a-1} y_{\bar{k}'_1} \cdots \hat{y_{\bar{k}'_a}}\cdots y_{\bar{k}'_r} \cdot y &\text{if }j=k'_a \in K'. \end{cases}
\]
So the second sum above becomes
\[
  \sum_K \sum_{j\not\in K} (-1)^{m+a-1} a_{jm} \Pf_K(A) y_{\bar{k}''_1} \cdots y_{\bar{k}''_{r-1}} \cdot y,
\]
where $K'' = K'\setminus j$.  Combining the two and using Laplace expansion yields the claimed formula for $x_\emptyset$.

The general case is similar.  Using \eqref{e.alt-basis}, it is equivalent to show
\[
 x_{i_1}\cdots x_{i_r}\cdot x_\emptyset = (-1)^{i_1+\cdots+i_r} \mathop{\sum_{K\subseteq [n]}}_{K \text{ even}} \Pf_K(A(I))\, y_K,
\]
and we do this by induction on $r$.  Let $I=\{i_1<\cdots<i_r\}$ and consider $i>i_r$.  We compute:
\begin{align*}
  x_{i_1} \cdots x_{i_r} \cdot x_i \cdot x_\emptyset &= (-1)^{r} x_{i} \cdot x_{i_1} \cdots x_{i_r}\cdot x_\emptyset \\
  &= (-1)^r \left( \sum_{j} b_{ji} y_j \right)(-1)^{i_1+\cdots+i_r} \sum_{K} \Pf_K(A(I)) y_K \\
  &= (-1)^{r+i_1+\cdots+i_r} \sum_{j}\sum_{K \not\ni j}  b_{ji}\Pf_K(A(I)) y_{j}\cdot y_K \\
  &= \cdots \\
  &= (-1)^{i+i_1+\cdots+i_r} \sum_K \Pf_K( A(I\cup\{i\}) ) y_K,
\end{align*}
as desired.
\end{proof}

\end{document}